\documentclass[12pt,bezier]{article}

\usepackage{amssymb,amsmath,amsthm,latexsym,tikz,url}
\newtheorem{theorem}{Theorem}[section]
\newtheorem{corollary}[theorem]{Corollary}
\newtheorem{definition}[theorem]{Definition}
\newtheorem{conjecture}[theorem]{Conjecture}
\newtheorem{lemma}[theorem]{Lemma}
\newtheorem{remark}[theorem]{Remark}

\newcommand\blfootnote[1]{
	\begingroup
	\renewcommand\thefootnote{}\footnote{#1}
	\addtocounter{footnote}{-1}
	\endgroup}
\newcommand{\z}{{\bf z}}
\newcommand{\y}{{\bf y}}
\newcommand{\x}{{\bf x}}
\newcommand{\1}{{\bf 1}}
\newcommand{\G}{\Gamma}
\newcommand{\g}{{\cal G}}
\newcommand{\h}{{\cal H}}
\tikzstyle{vertex}=[circle, draw, inner sep=0pt, minimum size=3pt]
\newcommand{\vertex}{\node[vertex]}
\usepackage{float}
\usepackage{subfig}
\usepackage{caption}

\begin{document}

\title{Quartic Graphs with Minimum Spectral Gap}
\author{ Maryam Abdi$^{\,\rm a}$ \quad  Ebrahim Ghorbani$^{\,\rm a,b,}$\thanks{Corresponding author}
\\[.3cm]
{\sl\normalsize $^{\rm a}$Department of Mathematics, K. N. Toosi University of Technology,}\\
{\sl\normalsize P. O. Box 16765-3381, Tehran, Iran}\\
{\sl\normalsize $^{\rm b}$Department of Mathematics, University of Hamburg, Bundesstra\ss{}e 55 (Geomatikum),}\\
{\sl\normalsize 20146 Hamburg, Germany }}
\maketitle
\blfootnote{{\em E-mail Addresses}: {\tt m.abdi@email.kntu.ac.ir, ebrahim.ghorbani@uni-hamburg.de}}

\begin{abstract}
Aldous and Fill conjectured that the maximum relaxation time for the random walk on a connected regular graph with $n$ vertices is $(1+o(1)) \frac{3n^2}{2\pi^2}$. This conjecture can be rephrased in terms of the spectral gap as follows:
the spectral gap (algebraic connectivity) of a connected  $k$-regular graph on $n$ vertices is at least
$(1+o(1))\frac{2k\pi^2}{3n^2}$, and the bound is attained for at least one value of $k$.
We determine the structure of connected quartic graphs on $n$ vertices with minimum spectral gap which enable us to show that the minimum spectral gap of connected quartic graphs on $n$ vertices is $(1+o(1))\frac{4\pi^2}{n^2}$.
From this result, the  Aldous--Fill conjecture follows for $k=4$.
 \vspace{4mm}
 	
\noindent {\bf AMS Mathematics Subject Classification\,(2010):}   05C50, 60G50
\end{abstract}

\section{Introduction}

All graphs we consider are simple, that is undirected graphs without loops or multiple edges.
The difference between the two largest eigenvalues of the adjacency matrix of a graph $G$ is called the {\em spectral gap} of $G$.
If $G$ is a regular graph, then its spectral gap is equal to the second smallest eigenvalue of its Laplacian matrix  and
known as {\em algebraic connectivity}.

In 1976,  Bussemaker, \v Cobelji\'c, Cvetkovi\'c, and Seidel (\cite{Bussemaker},  see also \cite{Bussemaker2}), by means of a computer search, found  all non-isomorphic connected cubic graphs with $n \leq 14$ vertices.  They observed that when the algebraic connectivity is small, the graph is  `long' (meaning roughly that their diameter is as large as possible; this was later on justified in \cite{Imrich}). Indeed, as the algebraic connectivity decreases, both connectivity and girth decrease and diameter increases.  Based on these results, L. Babai (see \cite{Guiduli}) made a conjecture that described the structure of the connected cubic graph with minimum algebraic connectivity.
Guiduli \cite{Guiduli} (see also \cite{GuiduliThesis}) proved that the cubic graph with minimum algebraic connectivity must look like a path, built from specific blocks.
The result of Guiduli was improved as follows confirming the Babai's conjecture.

\begin{theorem}[Brand, Guiduli, and Imrich \cite{Imrich}]\label{thm:cubicBGI}
	Among all connected cubic graphs on $n$ vertices,  $n \geq 10$, the cubic graph depicted in Figure~\ref{fig:Gn} is the unique graph with minimum algebraic connectivity.
\end{theorem}
\begin{figure}[h!]
	\centering
	\begin{tikzpicture}
	\vertex[fill] (1) at (0,-.5) [] {};
	\vertex[fill] (2) at (0,.5) [] {};
	\vertex[fill] (3) at (1,-.5) [] {};
	\vertex[fill] (4) at (1,.5) [] {};
	\vertex[fill] (5) at (1.5,0) [] {};
	\vertex[fill] (6) at (2,0) [] {};
	\vertex[fill] (7) at (2.5,.5) [] {};
	\vertex[fill] (8) at (2.5,-.5) [] {};
	\vertex[fill] (9) at (3,0) [] {};
	\vertex[fill] (10) at (3.5,0) [] {};
	\vertex[fill] (11) at (4,.5) [] {};
	\vertex[fill] (12) at (4,-.5) [] {};
	\vertex[fill] (13) at (4.5,0) [] {};
	\vertex[fill] (22) at (6.4,0) [] {};
	\vertex[fill] (23) at (6.9,.5) [] {};
	\vertex[fill] (24) at (6.9,-.5) [] {};
	\vertex[fill] (25) at (7.4,0) [] {};
	\vertex[fill] (26) at (7.9,0) [] {};
	\vertex[fill] (27) at (8.4,.5) [] {};
	\vertex[fill] (28) at (8.4,-.5) [] {};
	\vertex[fill] (29) at (8.9,0) [] {};
	\vertex[fill] (34) at (10.9,-.5) [] {};
	\vertex[fill] (33) at (10.9,.5) [] {};
	\vertex[fill] (32) at (9.9,-.5) [] {};
	\vertex[fill] (31) at (9.9,.5) [] {};
	\vertex[fill] (30) at (9.4,0) [] {};
	\tikzstyle{vertex}=[circle, draw, inner sep=0pt, minimum size=0pt]
	\vertex[fill] (14) at (4.8,0) [] {};
	\vertex[fill] (15) at (5.2,0) [] {};
	\vertex[fill] (16) at (5.3,0) [] {};
	\vertex[fill] (17) at (5.4,0) [] {};
	\vertex[fill] (18) at (5.5,0) [] {};
	\vertex[fill] (19) at (5.6,0) [] {};
	\vertex[fill] (20) at (5.7,0) [] {};
	\vertex[fill] (21) at (6.1,0) [] {};
	\path
	(1) edge (2)
	(1) edge (3)
	(1) edge (4)
	(2) edge (3)
	(2) edge (4)
	(3) edge (5)
	(4) edge (5)
	(5) edge (6)
	(6) edge (7)
	(6) edge (8)
	(7) edge (8)
	(7) edge (9)
	(8) edge (9)
	(9) edge (10)
	(10) edge (11)
	(10) edge (12)
	(11) edge (12)
	(11) edge (13)
	(12) edge (13)
	(13) edge (14)
	(15) edge (16)
	(17) edge (18)
	(19) edge (20)
	(21) edge (22)
	(22) edge (23)
	(22) edge (24)
	(23) edge (24)
	(23) edge (25)
	(24) edge (25)
	(25) edge (26)
	(26) edge (27)
	(26) edge (28)
	(27) edge (28)
	(27) edge (29)
	(28) edge (29)
	(29) edge (30)
	(30) edge (31)
	(30) edge (32)
	(31) edge (33)
	(31) edge (34)
	(32) edge (33)
	(32) edge (34)
	(33) edge (34);
	\end{tikzpicture}
	\vspace{.2cm} \\
	\begin{tikzpicture}
	\vertex[fill] (1) at (0,-.5) [] {};
	\vertex[fill] (2) at (0,.5) [] {};
	\vertex[fill] (3) at (1,-.5) [] {};
	\vertex[fill] (4) at (1,.5) [] {};
	\vertex[fill] (5) at (1.5,0) [] {};
	\vertex[fill] (6) at (2,0) [] {};
	\vertex[fill] (7) at (2.5,.5) [] {};
	\vertex[fill] (8) at (2.5,-.5) [] {};
	\vertex[fill] (9) at (3,0) [] {};
	\vertex[fill] (10) at (3.5,0) [] {};
	\vertex[fill] (11) at (4,.5) [] {};
	\vertex[fill] (12) at (4,-.5) [] {};
	\vertex[fill] (13) at (4.5,0) [] {};
	\vertex[fill] (22) at (6.4,0) [] {};
	\vertex[fill] (23) at (6.9,.5) [] {};
	\vertex[fill] (24) at (6.9,-.5) [] {};
	\vertex[fill] (25) at (7.4,0) [] {};
	\vertex[fill] (26) at (7.9,0) [] {};
	\vertex[fill] (27) at (8.4,.5) [] {};
	\vertex[fill] (28) at (8.4,-.5) [] {};
	\vertex[fill] (29) at (8.9,0) [] {};
	\vertex[fill] (34) at (10.9,-.5) [] {};
	\vertex[fill] (33) at (10.9,.5) [] {};
	\vertex[fill] (32) at (9.9,-.5) [] {};
	\vertex[fill] (31) at (9.9,.5) [] {};
	\vertex[fill] (30) at (9.4,0) [] {};
	\vertex[fill] (35) at (11.9,-.5) [] {};
	\vertex[fill] (36) at (11.9,.5) [] {};
	\tikzstyle{vertex}=[circle, draw, inner sep=0pt, minimum size=0pt]
	\vertex[fill] (14) at (4.8,0) [] {};
	\vertex[fill] (15) at (5.2,0) [] {};
	\vertex[fill] (16) at (5.3,0) [] {};
	\vertex[fill] (17) at (5.4,0) [] {};
	\vertex[fill] (18) at (5.5,0) [] {};
	\vertex[fill] (19) at (5.6,0) [] {};
	\vertex[fill] (20) at (5.7,0) [] {};
	\vertex[fill] (21) at (6.1,0) [] {};
	\path
	(1) edge (2)
	(1) edge (3)
	(1) edge (4)
	(2) edge (3)
	(2) edge (4)
	(3) edge (5)
	(4) edge (5)
	(5) edge (6)
	(6) edge (7)
	(6) edge (8)
	(7) edge (8)
	(7) edge (9)
	(8) edge (9)
	(9) edge (10)
	(10) edge (11)
	(10) edge (12)
	(11) edge (12)
	(11) edge (13)
	(12) edge (13)
	(13) edge (14)
	(15) edge (16)
	(17) edge (18)
	(19) edge (20)
	(21) edge (22)
	(22) edge (23)
	(22) edge (24)
	(23) edge (24)
	(23) edge (25)
	(24) edge (25)
	(25) edge (26)
	(26) edge (27)
	(26) edge (28)
	(27) edge (28)
	(27) edge (29)
	(28) edge (29)
	(29) edge (30)
	(30) edge (31)
	(30) edge (32)
	(31) edge (32)
	(31) edge (33)
	(32) edge (34)
	(33) edge (35)
	(34) edge (36)
	(33) edge (36)
	(34) edge (35)
	(35) edge (36);
	\end{tikzpicture}
	\caption{The cubic graph with minimum spectral gap on $n\equiv2\pmod4$ and $n\equiv0\pmod4$ vertices, respectively}
	\label{fig:Gn}
\end{figure}
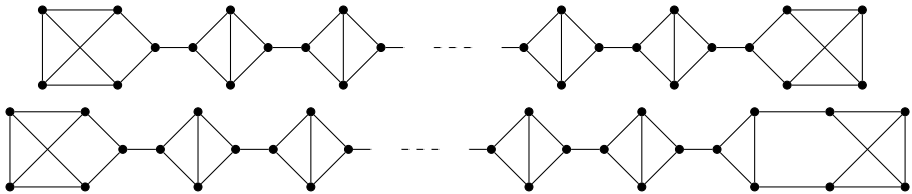

The {\em relaxation time} of the random walk on a graph $G$  is defined by $\tau=1/(1-\eta_2)$, where $\eta_2$ is the second largest eigenvalue of
the {\em transition matrix} of $G$, that is the matrix $\Delta^{-1} A$ in which $\Delta$ and $A$ are the diagonal matrix of vertex degrees and the adjacency matrix of $G$, respectively.
A central problem in the study of random walks is to determine the {\em mixing time}, a measure of how fast the
random walk converges to the stationary distribution. As seen throughout the literature \cite{aldous2002reversible,chung}, the  relaxation time is the primary term controlling mixing time. Therefore, relaxation time is directly associated with the rate of convergence  of the random walk.

Our main motivation in this work is the following conjecture on the maximum relaxation time of the random walk in regular graphs.

\begin{conjecture}[{Aldous and Fill \cite[p.~217]{aldous2002reversible}}]\label{conj:A-F} \rm
	Over all connected regular graphs on $n$ vertices, $\max \tau =(1+o(1)) \frac{3n^2}{2\pi^2}$.
\end{conjecture}

 For a graph $G$,  $L(G)=\Delta-A$ is its  Laplacian matrix. The second smallest eigenvalue of $L(G)$, i.e. the  algebraic connectivity of $G$ is denoted by  $\mu(G)$.
When  $G$ is regular, of degree $k$ say, then its transition matrix is $\frac1k A$ and its Laplacian is  $kI-A$.
It is then seen that the relaxation time of $G$ is equal to $k/\mu(G)$. Also as $G$ is regular, $\mu(G)$ is the same as its spectral gap.
So within the family of $k$-regular graphs, maximizing the relaxation time is equivalent to minimizing the spectral gap. 
More precisely, we have the following rephrasement  of the Aldous--Fill conjecture.

\begin{conjecture}\rm  The spectral gap (algebraic connectivity) of a connected   $k$-regular graph on $n$ vertices is at least $(1+o(1))\frac{2k\pi^2}{3n^2}$, and the bound is attained at least for one value of $k$.
\end{conjecture}

It is worth mentioning that in \cite{actt}, it is proved that the maximum relaxation time for the random walk on a  connected graph on $n$ vertices is $(1 +o(1))\frac{n^3}{54}$, settling another conjecture by Aldous and Fill (\cite[p.~216]{aldous2002reversible}).

Abdi, Ghobani and Imrich \cite{Abdi} proved that the algebraic connectivity of the graphs of Theorem~\ref{thm:cubicBGI} is $(1+o(1))\frac{2\pi^2}{n^2}$, implying that this is indeed the minimum spectral gap  of connected cubic graphs of order $n$. This settled the  Aldous--Fill  conjecture for $k=3$.
As the next case of the Aldous--Fill conjecture and as a continuation of Babai's conjecture, we consider quartic, i.e.~$4$-regular graphs and investigate {\em minimal} quartic graphs of a given order $n$, that is,  the  graphs which attain the minimum spectral gap among all conncetd quartic graphs of order $n$. In \cite{Abdi}, it was shown that similar to the cubic case, minimal quartic graphs have a path-like structure (see Figure~\ref{fig:path-like}) with specified blocks (see Theorem~\ref{thm:quarticOLD} below). However, the precise description  of minimal quartic graphs was left as a conjecture given below.

\begin{figure}[h!]
	\centering
	\begin{tikzpicture}
	\tikzstyle{vertex}=[draw, inner sep=0pt, minimum size=0pt]
	\vertex[fill] (1) at (0,0) [label=left:\tiny{}] {};
	\vertex[fill] (2) at (.5,.5) [] {};
	\vertex[fill] (3) at (.5,-.5) [] {};
	\vertex[fill] (4) at (1.5,.5) [] {};
	\vertex[fill] (5) at (1.5,-.5) [] {};
	\vertex[fill] (6) at (2,0) [] {};
	\vertex[fill] (7) at (2.5,.5) [] {};
	\vertex[fill] (8) at (2.5,-.5) [] {};
	\vertex[fill] (9) at (3.5,.5) [] {};
	\vertex[fill] (10) at (3.5,-.5) [] {};
	\vertex[fill] (11) at (4,0) [] {};
	\vertex[fill] (12) at (4.5,.5) [] {};
	\vertex[fill] (13) at (4.5,-.5) [] {};
	\vertex[fill] (14) at (5.5,.5) [] {};
	\vertex[fill] (15) at (5.5,-.5) [] {};
	\vertex[fill] (16) at (6,0) [] {};
	\vertex[fill] (20) at (7.6,0) [] {};
	\vertex[fill] (21) at (8.1,.5) [] {};
	\vertex[fill] (22) at (8.1,-.5) [] {};
	\vertex[fill] (23) at (9.1,.5) [] {};
	\vertex[fill] (24) at (9.1,-.5) [] {};
	\vertex[fill] (25) at (9.6,0) [] {};
	\tikzstyle{vertex}=[circle, draw, inner sep=0pt, minimum size=1pt]
	\vertex[fill] (17) at (6.6,0) [] {};
	\vertex[fill] (18) at (6.8,0) [] {};
	\vertex[fill] (19) at (7,0) [] {};
	\tikzstyle{vertex}=[circle, draw, inner sep=0pt, minimum size=0pt]
	\vertex (s) at (6.2,.2) [label=right:$$] {};
	\vertex (ss) at (6.2,-.2) [label=right:$$] {};
	\vertex (sss) at (7.4,.2) [label=right:$$] {};
	\vertex (ssss) at (7.4,-.2) [label=right:$$] {};
	\path[draw,thick]
	(1) edge (2)
	(1) edge (3)
	(3) edge (5)
	(2) edge (4)
	(4) edge (6)
	(5) edge (6)
	(6) edge (7)
	(6) edge (8)
	(7) edge (9)
	(8) edge (10)
	(10) edge (11)
	(9) edge (11)
	(11) edge (12)
	(11) edge (13)
	(12) edge (14)
	(13) edge (15)
	(14) edge (16)
	(15) edge (16)
	(16) edge (ss)
	(16) edge (s)
	(sss) edge (20)
	(ssss) edge (20)
	(20) edge (21)
	(20) edge (22)
	(21) edge (23)
	(22) edge (24)
	(24) edge (25)
	(23) edge (25);
	\end{tikzpicture}
	\caption{The path-like structure}\label{fig:path-like}
\end{figure}
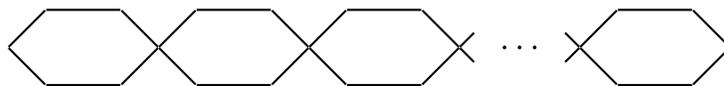

\begin{definition}\rm For any $n\ge11$, define a quartic graph $\g_n$ as follows.
Let $m$ and $r\le4$ be non-negative integers such that $n-11=5m+r$.
Then $\g_n$ consists of $m$ middle blocks $M_0$ and each end block is one of $D_0,D_1,D_2,D_4$ of Figure~\ref{fig:BlockMin}. If $r=0$, then both end blocks are $D_0$. If $r=1$, then  the end blocks are $D_0$ and $D_1$. If $r=2$, then both end blocks are $D_1$. If $r=3$, then  the end blocks are $D_1$ and $D_2$.  Finally, if  $r=4$, then  the end blocks are $D_0$ and $D_4$. (Note that right end blocks are the mirror images of one of $D_0,D_1,D_2,D_4$.)
\end{definition}
\begin{figure}[h!]
	\captionsetup[subfigure]{labelformat=empty}
	\centering
	\subfloat[$M_0$]{\begin{tikzpicture}[scale=0.9]
		\vertex[fill] (r) at (0,0) [] {};
		\vertex[fill] (r1) at (.5,.5) [] {};
		\vertex[fill] (r2) at (.5,-.5) [] {};
		\vertex[fill] (r3) at (1.5,.5) [] {};
		\vertex[fill] (r4) at (1.5,-.5) [] {};
		\vertex[fill] (r5) at (2,0) [][] {};
		\path
		(r) edge (r1)
		(r) edge (r2)
		(r1) edge (r2)
		(r1) edge (r3)
		(r1) edge (r4)
		(r2) edge (r3)
		(r2) edge (r4)
		(r3) edge (r4)
		(r5) edge (r4)
		(r3) edge (r5);
		\end{tikzpicture}}
	\quad
	\subfloat[$D_0$]{\begin{tikzpicture}[scale=0.9]
		\vertex[fill] (r) at (0,0) [] {};
		\vertex[fill] (r1) at (.5,.5) [] {};
		\vertex[fill] (r2) at (.5,-.5) [] {};
		\vertex[fill] (r3) at (1.5,.5) [] {};
		\vertex[fill] (r4) at (1.5,-.5) [] {};
		\vertex[fill] (r5) at (2,0) [] {};
		\path
		(r) edge (r1)
		(r) edge (r2)
		(r) edge (r3)
		(r) edge (r4)
		(r1) edge (r2)
		(r1) edge (r3)
		(r1) edge (r4)
		(r2) edge (r3)
		(r2) edge (r4)
		(r5) edge (r4)
		(r5) edge (r3);
		\end{tikzpicture}}
	\quad
	\subfloat[$D_1$]{\begin{tikzpicture}[scale=0.9]
		\vertex[fill] (r1) at (0,.8) [] {};
		\vertex[fill] (r2) at (.8,1) [] {};
		\vertex[fill] (r3) at (0,0) [] {};
		\vertex[fill] (r4) at (.8,-.2) [] {};
		\vertex[fill] (r5) at (1.6,.8) [] {};
		\vertex[fill] (r6) at (1.6,0) [] {};
		\vertex[fill] (r7) at (2.3,.4) [] {};
		\path
		(r5) edge (r2)
		(r1) edge (r2)
		(r1) edge (r5)
		(r1) edge (r3)
		(r1) edge (r4)
		(r2) edge (r3)
		(r2) edge (r4)
		(r3) edge (r6)
		(r4) edge (r3)
		(r4) edge (r6)
		(r5) edge (r6)
		(r7) edge (r6)
		(r5) edge (r7);
		\end{tikzpicture}}
	\quad
	\subfloat[$D_2$]{\begin{tikzpicture}[scale=0.9]
		\vertex[fill] (r1) at (0,0) [] {};
		\vertex[fill] (r2) at (0,.8) [] {};
		\vertex[fill] (r3) at (.8,1) [] {};
		\vertex[fill] (r4) at (.8,-.2) [] {};
		\vertex[fill] (r5) at (1.3,.4) [] {};
		\vertex[fill] (r6) at (2,.8) [] {};
		\vertex[fill] (r7) at (2,0) [] {};
		\vertex[fill] (r8) at (2.7,.4) [] {};
		\path
		(r5) edge (r2)
		(r1) edge (r2)
		(r1) edge (r5)
		(r1) edge (r3)
		(r1) edge (r4)
		(r2) edge (r3)
		(r2) edge (r4)
		(r3) edge (r6)
		(r4) edge (r3)
		(r4) edge (r7)
		(r5) edge (r7)
		(r5) edge (r6)
		(r6) edge (r7)
		(r7) edge (r8)
		(r6) edge (r8)  ;
		\end{tikzpicture}}
	\quad
	\subfloat[$D_4$]{\begin{tikzpicture}[scale=0.9]
		\vertex[fill] (r1) at (0,0) [] {};
		\vertex[fill] (r3) at (.5,.5) [] {};
		\vertex[fill] (r2) at (.5,-.5) [] {};
		\vertex[fill] (r5) at (1.5,.5) [] {};
		\vertex[fill] (r4) at (1.5,-.5) [] {};
		\vertex[fill] (r6) at (2.3,-.5) [] {};
		\vertex[fill] (r7) at (2.3,.5) [] {};
		\vertex[fill] (r8) at (3.3,-.5) []{};
		\vertex[fill] (r9) at (3.3,.5) [] {};
		\vertex[fill] (r10) at (3.8,0) [] {};
		\path
		(r1) edge (r2)
		(r1) edge (r3)
		(r1) edge (r4)
		(r1) edge (r5)
		(r2) edge (r3)
		(r2) edge (r4)
		(r2) edge (r5)
		(r3) edge (r4)
		(r3) edge (r5)
		(r6) edge (r4)
		(r5) edge (r7)
		(r6) edge (r7)
		(r6) edge (r8)
		(r6) edge (r9)
		(r8) edge (r7)
		(r7) edge (r9)
		(r8) edge (r9)
		(r8) edge (r10)
		(r9) edge (r10);
		\end{tikzpicture}}
	\caption{The blocks of the quartic graphs $\g_n$}\label{fig:BlockMin}
\end{figure}
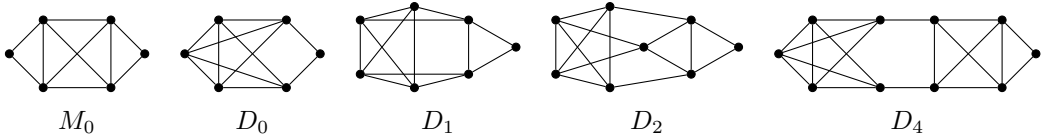

\begin{conjecture}[\cite{Abdi}]\label{conj:MinQuartic}\rm  For any $n\ge11$, the graph $\g_n$ is the unique minimal quartic graph of order $n$.
\end{conjecture}

We `almost' prove Conjecture~\ref{conj:MinQuartic}  (see Theorem~\ref{thm:quarticNEW} below) by showing that in a minimal quartic graph
any middle block is $M_0$ and any end block is one of $D_0,D_1,D_2,D_4$, or possibly one additional block $D_3$; see Figure~\ref{fig:D3}.
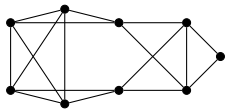
\begin{figure}[h!]
	\centering
	\begin{tikzpicture}[scale=0.9]
	\vertex[fill] (r1) at (0,1) [] {};
	\vertex[fill] (r2) at (.8,1.2) [] {};
	\vertex[fill] (r3) at (0,0) [] {};
	\vertex[fill] (r4) at (.8,-.2) [] {};
	\vertex[fill] (r5) at (1.6,1) [] {};
	\vertex[fill] (r6) at (1.6,0) [] {};
	\vertex[fill] (r7) at (2.6,1) [] {};
	\vertex[fill] (r8) at (2.6,0) [] {};
	\vertex[fill] (r9) at (3.1,.5) [] {};
	\path
	(r5) edge (r2)
	(r1) edge (r2)
	(r1) edge (r5)
	(r1) edge (r3)
	(r1) edge (r4)
	(r2) edge (r3)
	(r2) edge (r4)
	(r3) edge (r6)
	(r4) edge (r3)
	(r4) edge (r6)
	(r5) edge (r8)
	(r5) edge (r7)
	(r6) edge (r8)
	(r6) edge (r7)
	(r7) edge (r8)
	(r9) edge (r7)
	(r9) edge (r8);
	\end{tikzpicture}
	\caption{The block $D_3$}
	\label{fig:D3}
\end{figure}
 This result is enough  to prove that the minimum algebraic connectivity of connected quartic graphs of order $n$ is $(1+o(1))\frac{4\pi^2}{n^2}$  (see Corollary~\ref{corollary} below).
 This shows that, although the Aldous--Fill conjecture
does hold for $k=4$, the bound given by it (that is $(1+o(1))\frac{8\pi^2}{3n^2}$) is not tight for $4$-regular graphs (in contrast to $3$-regular graphs where the
bound is tight).

The outline of the paper is as follows. In Section~\ref{sec:g_n}, we determine the algebraic connectivity of the graphs $\g_n$ as well as those with end block $D_3$, and in Section~\ref{sec:structure},
 we prove  that indeed the minimal quartic graphs belong to this family.

\section{Minimum algebraic connectivity of quartic graphs} \label{sec:g_n}

In this section, we show that the algebraic connectivity of the quartic graphs of order $n$ with a path-like structure
whose middle blocks are  all $M_0$ and end blocks are from $D_0,\ldots,D_4$ is  $(1+o(1))\frac{4\pi^2}{n^2}$.
In the next section, we will prove that minimal quartic graphs belong to the family described above.
From these two results, it readily follows that $(1+o(1))\frac{4\pi^2}{n^2}$ is indeed the minimum algebraic connectivity of connected quartic graphs of order $n$.

Let $G$ be a graph of order $n$  and $L(G)=\Delta-A$ be its Laplacian matrix.
It is well-known that $L(G)\1^\top=\bf0^\top$, where $\bf1$ is the all-$1$ vector,\footnote{We treat vectors as ``row vectors.''}
and $0$ is in fact the smallest eigenvalue of $L(G)$.
The second smallest eigenvalue $\mu(G)$ of $L(G)$ is the algebraic connectivity of $G$. %, denoted by  , where
The reason for this name is the well-known fact that $\mu(G)>0$ if and only if $G$ is connected.
 It is also known that
\begin{equation}\label{eq}
\mu(G)=\min_{\x\ne\bf0,\,\x\perp\bf1}\frac{\x L(G)\x^\top}{\x\x^\top}.
\end{equation}
 An eigenvector corresponding to $\mu(G)$ is known as a {\em Fiedler vector} of $G$.
In passing, we note that if $\x=(x_1,\ldots,x_n)$, then
\begin{equation}\label{eq2}
\x L(G)\x^\top=\sum_{ij\in E(G)}(x_i-x_j)^2,
\end{equation}
where $E(G)$ is the edge set of $G$. Recall that if $\x$ is an eigenvector corresponding to $\mu$, then for any vertex $i$ with degree $d_i$,
\begin{equation}\label{eq:eigenvector}
\mu x_i=d_ix_i-\sum_{j:\,ij\in E(G)}x_j.
\end{equation}
We refer to \eqref{eq:eigenvector} as the {\em eigen-equation}.

 By \eqref{eq}, $\mu(G)$ is determined by the vectors orthogonal to $\1$. The other vectors, however, can also be used to obtain potentially good upper bounds for $\mu(G)$. This is formulated in the following key lemma.
\begin{lemma}\label{remark:delta}
Let $G$ be a graph of order $n$, $\x$ an arbitrary vector of length $n$ which is not a multiple of $\1$ and $\delta:=\x\1^\top$. Then
$$\mu(G)\le\frac{\x L(G)\x^\top}{\|\x\|^2-\frac{\delta^2}n}.$$
\end{lemma}
\begin{proof}
Let  $\y=\x-\frac\delta n\1$. Then $\y\perp\1$ and so
$\mu(G)\le\frac{\y L(G)\y^\top}{\|\y\|^2}$.
As $L(G)\1^\top=\bf0^\top$, we have $\y L(G)\y^\top=\x L(G)\x^\top$.
Furthermore,
$$\|\y\|^2=\|\x\|^2-\frac{2\delta} n\x\1^\top+\frac{\delta^2}n=\|\x\|^2-\frac{\delta^2}n,$$
and thus the result follows.
\end{proof}

Let $\Pi=\{C_1,  \ldots , C_p\}$ be a partition of the vertex set $V(G)$.
Then $\Pi$ is called an {\em equitable partition} of $G$ if for all $i,j$ (possibly $i=j$),
the number of neighbors of any vertex of $C_i$ in $C_j$ only depends on $i$ and $j$.

\begin{definition}\label{def:fit}\rm
Let $D$ and $D'$ be two graphs. We say that $D'$ {\em fits} $D$ if
\begin{itemize}
	\item[(i)] $D$ has an equitable partition  $\Pi=\{C_1,  \ldots , C_p\}$ and $V(D')$ has a partion  $\Pi'=\{C'_1,  \ldots , C'_p\}$ (not necessarily equitable);
	\item[(ii)] 	$|C_i|\le|C'_i|$ for $i=1,\ldots,p$;
	\item[(iii)] for $1\le i<j\le p$, the number of edges between $C_i$ and $C_j$ in $D$ is the same as the number of edges between $C'_i$ and $C'_j$ in $D'$.
\end{itemize}
\end{definition}

In Table~\ref{tab}, we illustrate the blocks from $D_1,\ldots,D_4$ which fit each of  $D_0,\ldots,D_3$.
For instance, the first row shows that each of $D_1,\ldots,D_4$ fit $D_0$.
The corresponding partitions $\Pi$ and $\Pi'$ are specified by dashed lines.
\begin{table}[h!]
	\centering
	\begin{tabular}{l|c}
		\hline
		The partition $\Pi$ &   The partition $\Pi'$ \\ \hline
		$D_0$\quad$\begin{array}{c}\begin{tikzpicture}[scale=0.7]
		\tikzstyle{vertex}=[circle, draw, inner sep=0pt, minimum size=2pt]
		\vertex[fill] (r) at (0,0) [] {};
		\vertex[fill] (r1) at (.5,.5) [] {};
		\vertex[fill] (r2) at (.5,-.5) [] {};
		\vertex[fill] (r3) at (1.5,.5) [] {};
		\vertex[fill] (r4) at (1.5,-.5) [] {};
		\vertex[fill] (r5) at (2,0) [] {};
		\path
		(r) edge (r1)
		(r) edge (r2)
		(r) edge (r3)
		(r) edge (r4)
		(r1) edge (r2)
		(r1) edge (r3)
		(r1) edge (r4)
		(r2) edge (r3)
		(r2) edge (r4)
		(r5) edge (r4)
		(r5) edge (r3);
		\draw[densely dashed]  plot[smooth, tension=1] coordinates {(-.2,0) (.6,.7) (.6,-.7) (-.2,-.1) };
		\draw[densely dashed] (1.5,0) ellipse (.15 and .8);
		\draw[densely dashed] (2,0) ellipse (.15 and .28);
		\end{tikzpicture} \end{array}$
		&
		$\begin{array}{cccc}
		\begin{tikzpicture}[scale=0.7]
		\tikzstyle{vertex}=[circle, draw, inner sep=0pt, minimum size=2pt]
		\vertex[fill] (r1) at (0,.8) [] {};
		\vertex[fill] (r2) at (.8,1) [] {};
		\vertex[fill] (r3) at (0,0) [] {};
		\vertex[fill] (r4) at (.8,-.2) [] {};
		\vertex[fill] (r5) at (1.6,.8) [] {};
		\vertex[fill] (r6) at (1.6,0) [] {};
		\vertex[fill] (r7) at (2.3,.4) [] {};
		\path
		(r1) edge (r2)
		(r1) edge (r3)
		(r2) edge (r3)
		(r4) edge (r6)
		(r5) edge (r6);
		\path
		(r5) edge (r2)		
		(r1) edge (r5)	
		(r1) edge (r4)	
		(r2) edge (r4)
		(r3) edge (r6)
		(r4) edge (r3)
		(r7) edge (r6)
		(r5) edge (r7)  ;
		\draw[densely dashed]  plot[smooth, tension=1] coordinates {(-.2,.9) (1,1.1) (0,-.2) (-.2,.9) };
		\draw[densely dashed]  plot[smooth, tension=1] coordinates {(.6,-.3) (1.5,1) (1.7,-.1) (.6,-.3) };
		\draw[densely dashed] (2.3,.4) ellipse (.15 and .28);
		\end{tikzpicture}
		&
		\begin{tikzpicture}[scale=0.7]
		\tikzstyle{vertex}=[circle, draw, inner sep=0pt, minimum size=2pt]
		\vertex[fill] (r1) at (0,0) [] {};
		\vertex[fill] (r2) at (0,.8) [] {};
		\vertex[fill] (r3) at (.8,1) [] {};
		\vertex[fill] (r4) at (.8,-.2) [] {};
		\vertex[fill] (r5) at (1.3,.4) [] {};
		\vertex[fill] (r6) at (2,.8) [] {};
		\vertex[fill] (r7) at (2,0) [] {};
		\vertex[fill] (r8) at (2.7,.4) [] {};
		\path
		(r1) edge (r2)
		(r1) edge (r3)
		(r2) edge (r3)
		(r4) edge (r7)
		(r5) edge (r7)
		(r6) edge (r7)
		(r5) edge (r6);
		\path  	
		(r5) edge (r2)		
		(r1) edge (r5)		
		(r1) edge (r4)		
		(r2) edge (r4)
		(r3) edge (r6)
		(r4) edge (r3) 	
		(r7) edge (r8)
		(r6) edge (r8)  ;
		\draw[densely dashed]  plot[smooth, tension=1] coordinates {(-.2,.9) (1,1.1) (0,-.2) (-.2,.9) };
		\draw[densely dashed]  plot[smooth, tension=1] coordinates {(.6,-.3)(1.1,.4) (2,1) (2.1,-.1) (.6,-.3) };
		\draw[densely dashed] (2.7,.4) ellipse (.15 and .28);
		\end{tikzpicture}
		&
		\begin{tikzpicture}[scale=0.7]
		\tikzstyle{vertex}=[circle, draw, inner sep=0pt, minimum size=2pt]
		\vertex[fill] (r1) at (0,1) [] {};
		\vertex[fill] (r2) at (.8,1.2) [] {};
		\vertex[fill] (r3) at (0,0) [] {};
		\vertex[fill] (r4) at (.8,-.2) [] {};
		\vertex[fill] (r5) at (1.6,1) [] {};
		\vertex[fill] (r6) at (1.6,0) [] {};
		\vertex[fill] (r7) at (2.6,1) [] {};
		\vertex[fill] (r8) at (2.6,0) [] {};
		\vertex[fill] (r9) at (3.1,.5) [] {};
		\tikzstyle{vertex}=[circle, draw, inner sep=0pt, minimum size=0pt]
		\vertex (s) at (3.3,.7) [] {};
		\vertex (ss) at (3.3,.3) [] {};
		\path
		(r1) edge (r2)
		(r2) edge (r3)
		(r1) edge (r3)
		(r4) edge (r6)
		(r5) edge (r8)
		(r5) edge (r7)
		(r6) edge (r8)
		(r6) edge (r7)
		(r7) edge (r8);
		\path
		(r5) edge (r2)		
		(r1) edge (r5)		
		(r1) edge (r4)		
		(r2) edge (r4)
		(r3) edge (r6)
		(r4) edge (r3)	
		(r9) edge (r7)
		(r9) edge (r8);
		\draw[densely dashed]  plot[smooth, tension=1] coordinates {(-.2,1.1) (1,1.25) (0,-.2) (-.2,1.1) };
		\draw[densely dashed] (.55,-.35)--(1.5,1.15)--(2.75,1.15)--(2.75,-.15)--(.55,-.35) ;
		\draw[densely dashed] (3.1,.5) ellipse (.15 and .28);
		\end{tikzpicture}
	&
		\begin{tikzpicture}[scale=0.7]
		\tikzstyle{vertex}=[circle, draw, inner sep=0pt, minimum size=2pt]
		\vertex[fill] (r1) at (0,0) [] {};
		\vertex[fill] (r3) at (.5,.5) [] {};
		\vertex[fill] (r2) at (.5,-.5) [] {};
		\vertex[fill] (r5) at (1.5,.5) [] {};
		\vertex[fill] (r4) at (1.5,-.5) [] {};
		\vertex[fill] (r6) at (2.3,-.5) [] {};
		\vertex[fill] (r7) at (2.3,.5) [] {};
		\vertex[fill] (r8) at (3.3,-.5) []{};
		\vertex[fill] (r9) at (3.3,.5) [] {};
		\vertex[fill] (r10) at (3.8,0) [] {};
		\path
		(r1) edge (r2)
		(r1) edge (r3)
		(r2) edge (r3)
		(r6) edge (r4)
		(r5) edge (r7)
		(r6) edge (r7)
		(r6) edge (r8)
		(r6) edge (r9)
		(r8) edge (r7)
		(r7) edge (r9)
		(r8) edge (r9);
		\path
		(r1) edge (r4)
		(r1) edge (r5)		
		(r2) edge (r4)
		(r2) edge (r5)
		(r3) edge (r4)
		(r3) edge (r5)	
		(r8) edge (r10)
		(r9) edge (r10);
		\draw[densely dashed]  plot[smooth, tension=1] coordinates {(-.2,0) (.6,.7) (.6,-.7) (-.2,0) };
		\draw[densely dashed]  (1.3,.65) rectangle (3.45,-.65) ;
		\draw[densely dashed] (3.8,0) ellipse (.15 and .28);
		\end{tikzpicture}\end{array}$
		\\ \hline\vspace{-.25cm}\\
		$D_1$\quad$\begin{array}{c}\begin{tikzpicture}[scale=0.7]
		\tikzstyle{vertex}=[circle, draw, inner sep=0pt, minimum size=2pt]
		\vertex[fill] (r1) at (0,.8) [] {};
		\vertex[fill] (r2) at (.8,1) [] {};
		\vertex[fill] (r3) at (0,0) [] {};
		\vertex[fill] (r4) at (.8,-.2) [] {};
		\vertex[fill] (r5) at (1.6,.8) [] {};
		\vertex[fill] (r6) at (1.6,0) [] {};
		\vertex[fill] (r7) at (2.3,.4) [] {};
		\path
		(r1) edge (r2)
		(r1) edge (r3)
		(r2) edge (r3)
		(r4) edge (r6)
		(r5) edge (r6);
		\path
		(r5) edge (r2)		
		(r1) edge (r5)	
		(r1) edge (r4)	
		(r2) edge (r4)
		(r3) edge (r6)
		(r4) edge (r3)
		(r7) edge (r6)
		(r5) edge (r7)  ;
		\draw[densely dashed]  (-.16,.95)-- (.95,1.2) -- (.95,-.4)--(-.16,-.1)--(-.16,.95) ;
		\draw[densely dashed] (1.6,.4) ellipse (.15 and .6);
		\draw[densely dashed] (2.3,.4) ellipse (.15 and .28);
		\end{tikzpicture}\end{array}$
		&  $\begin{array}{ccccc}
		\begin{tikzpicture}[scale=0.7]
		\tikzstyle{vertex}=[circle, draw, inner sep=0pt, minimum size=2pt]
		\vertex[fill] (r1) at (0,0) [] {};
		\vertex[fill] (r2) at (0,.8) [] {};
		\vertex[fill] (r3) at (.8,1) [] {};
		\vertex[fill] (r4) at (.8,-.2) [] {};
		\vertex[fill] (r5) at (1.3,.4) [] {};
		\vertex[fill] (r6) at (2,.8) [] {};
		\vertex[fill] (r7) at (2,0) [] {};
		\vertex[fill] (r8) at (2.7,.4) [] {};
		\path
		(r1) edge (r2)
		(r1) edge (r3)
		(r2) edge (r3)
		(r1) edge (r4)		
		(r2) edge (r4)
		(r4) edge (r3) 	
		(r5) edge (r7)
		(r6) edge (r7)
		(r5) edge (r6);
		\path  	
		(r5) edge (r2)		
		(r1) edge (r5)		
		(r4) edge (r7)
		(r3) edge (r6)	   	
		(r7) edge (r8)
		(r6) edge (r8)  ;
		\draw[densely dashed]  (-.16,.95)-- (.95,1.2) -- (.95,-.4)--(-.16,-.1)--(-.16,.95) ;
		\draw[densely dashed]  plot[smooth, tension=1] coordinates {(1.1,.4) (2.1,.9) (2.1,-.1) (1.1,.4) };
		\draw[densely dashed] (2.7,.4) ellipse (.15 and .28);
		\end{tikzpicture}
		&&
		\begin{tikzpicture}[scale=0.7]
		\tikzstyle{vertex}=[circle, draw, inner sep=0pt, minimum size=2pt]
		\vertex[fill] (r1) at (0,1) [] {};
		\vertex[fill] (r2) at (.8,1.2) [] {};
		\vertex[fill] (r3) at (0,0) [] {};
		\vertex[fill] (r4) at (.8,-.2) [] {};
		\vertex[fill] (r5) at (1.6,1) [] {};
		\vertex[fill] (r6) at (1.6,0) [] {};
		\vertex[fill] (r7) at (2.6,1) [] {};
		\vertex[fill] (r8) at (2.6,0) [] {};
		\vertex[fill] (r9) at (3.1,.5) [] {};
		\path
		(r1) edge (r2)
		(r2) edge (r3)
		(r1) edge (r3)
		(r2) edge (r4)
		(r4) edge (r3)	
		(r1) edge (r4)	
		(r5) edge (r8)
		(r5) edge (r7)
		(r6) edge (r8)
		(r6) edge (r7)
		(r7) edge (r8);
		\path
		(r5) edge (r2)		
		(r1) edge (r5)		     		
		(r3) edge (r6)
		(r4) edge (r6)
		(r9) edge (r7)
		(r9) edge (r8);
		\draw[densely dashed]  (-.16,1.1)-- (.95,1.4) -- (.95,-.4)--(-.16,-.1)--(-.16,1.1) ;
		\draw[densely dashed] (1.43,-.15)--(1.43,1.15)--(2.75,1.15)--(2.75,-.15)--(1.43,-.15) ;
		\draw[densely dashed] (3.1,.5) ellipse (.15 and .28);
		\end{tikzpicture}
	&&
	\begin{tikzpicture}[scale=0.7]
		\tikzstyle{vertex}=[circle, draw, inner sep=0pt, minimum size=2pt]
		\vertex[fill] (r1) at (0,0) [] {};
		\vertex[fill] (r3) at (.5,.5) [] {};
		\vertex[fill] (r2) at (.5,-.5) [] {};
		\vertex[fill] (r5) at (1.5,.5) [] {};
		\vertex[fill] (r4) at (1.5,-.5) [] {};
		\vertex[fill] (r6) at (2.3,-.5) [] {};
		\vertex[fill] (r7) at (2.3,.5) [] {};
		\vertex[fill] (r8) at (3.3,-.5) []{};
		\vertex[fill] (r9) at (3.3,.5) [] {};
		\vertex[fill] (r10) at (3.8,0) [] {};
		\path
		(r1) edge (r2)
		(r1) edge (r3)
		(r2) edge (r3)	        	
		(r6) edge (r7)
		(r6) edge (r8)
		(r6) edge (r9)
		(r8) edge (r7)
		(r7) edge (r9)
		(r8) edge (r9) 	
		(r6) edge (r4)
		(r1) edge (r5)				
		(r2) edge (r5)		
		(r3) edge (r5)		;
		\path	
		(r3) edge (r4)
		(r1) edge (r4)
		(r2) edge (r4)	
		(r5) edge (r7)	
		(r8) edge (r10)
		(r9) edge (r10);
		\draw[densely dashed]  plot[smooth, tension=1] coordinates {(-.2,0) (.6,.7)(1.63,.45) (.52,-.65) (-.2,0) };
		\draw[densely dashed]  (2.2,.65)--  (3.45,.65)--(3.45,-.65)--(1.2,-.65)--(2.2,.65) ;
		\draw[densely dashed] (3.8,0) ellipse (.15 and .28);	
		\end{tikzpicture}\end{array}$
		\\ \hline\vspace{-.25cm}\\
		$D_2$\quad$\begin{array}{c}\begin{tikzpicture}[scale=0.7]
		\tikzstyle{vertex}=[circle, draw, inner sep=0pt, minimum size=2pt]
		\vertex[fill] (r1) at (0,0) [] {};
		\vertex[fill] (r2) at (0,.8) [] {};
		\vertex[fill] (r3) at (.8,1) [] {};
		\vertex[fill] (r4) at (.8,-.2) [] {};
		\vertex[fill] (r5) at (1.3,.4) [] {};
		\vertex[fill] (r6) at (2,.8) [] {};
		\vertex[fill] (r7) at (2,0) [] {};
		\vertex[fill] (r8) at (2.7,.4) [] {};
		\path
		(r1) edge (r2)
		(r1) edge (r3)
		(r2) edge (r3)
		(r1) edge (r4)		
		(r2) edge (r4)
		(r4) edge (r3) 	
		(r5) edge (r7)
		(r6) edge (r7)
		(r5) edge (r6)	   	
		(r5) edge (r2)		
		(r1) edge (r5)		
		(r4) edge (r7)
		(r3) edge (r6)	   	
		(r7) edge (r8)
		(r6) edge (r8)  ;	
		\draw[densely dashed] (0,.4) ellipse (.15 and .6);
		\draw[densely dashed] (.8,.4) ellipse (.16 and .81);
		\draw[densely dashed] (1.3,.4) ellipse (.15 and .28);
		\draw[densely dashed] (2,.4) ellipse (.15 and .6);
		\draw[densely dashed] (2.7,.4) ellipse (.15 and .28);
		\end{tikzpicture}\end{array}$
		&
		$\begin{array}{c}\begin{tikzpicture}[scale=0.7]
		\tikzstyle{vertex}=[circle, draw, inner sep=0pt, minimum size=2pt]
		\vertex[fill] (r1) at (0,1) [] {};
		\vertex[fill] (r2) at (.8,1.2) [] {};
		\vertex[fill] (r3) at (0,0) [] {};
		\vertex[fill] (r4) at (.8,-.2) [] {};
		\vertex[fill] (r5) at (1.6,1) [] {};
		\vertex[fill] (r6) at (1.6,0) [] {};
		\vertex[fill] (r7) at (2.6,1) [] {};
		\vertex[fill] (r8) at (2.6,0) [] {};
		\vertex[fill] (r9) at (3.1,.5) [] {};
		\path
		(r1) edge (r2)
		(r4) edge (r3)	       	
		(r6) edge (r8)
		(r6) edge (r7)
		(r7) edge (r8);
		\path
		(r5) edge (r8)
		(r5) edge (r7)
		(r4) edge (r6)
		(r1) edge (r3)
		(r5) edge (r2)		
		(r1) edge (r5)		
		(r1) edge (r4)		
		(r2) edge (r4)
		(r3) edge (r6)
		(r2) edge (r3)
		(r9) edge (r7)
		(r9) edge (r8);
		\draw[densely dashed, rotate={15}] (.7,.94) ellipse (.6 and .2);
		\draw[densely dashed,rotate={-15}] (.4,0) ellipse (.6 and .2);
		\draw[densely dashed]  plot[smooth, tension=1] coordinates {(1.4,-.1)(2.52,1.16)(2.65,-.15)(1.43,-.1) };
		\draw[densely dashed] (3.1,.5) ellipse (.15 and .28);
		\draw[densely dashed] (1.6,1) ellipse (.15 and .28);
		\end{tikzpicture}\end{array}$
		\\ \hline\vspace{-.25cm}\\
		$D_3$\quad$\begin{array}{c}\begin{tikzpicture}[scale=0.7]
		\tikzstyle{vertex}=[circle, draw, inner sep=0pt, minimum size=2pt]
		\vertex[fill] (r1) at (0,1) [] {};
		\vertex[fill] (r2) at (.8,1.2) [] {};
		\vertex[fill] (r3) at (0,0) [] {};
		\vertex[fill] (r4) at (.8,-.2) [] {};
		\vertex[fill] (r5) at (1.6,1) [] {};
		\vertex[fill] (r6) at (1.6,0) [] {};
		\vertex[fill] (r7) at (2.6,1) [] {};
		\vertex[fill] (r8) at (2.6,0) [] {};
		\vertex[fill] (r9) at (3.1,.5) [] {};
		\path
		(r1) edge (r2)
		(r2) edge (r3)
		(r1) edge (r3)
		(r4) edge (r6)
		(r5) edge (r8)
		(r5) edge (r7)
		(r6) edge (r8)
		(r6) edge (r7)
		(r7) edge (r8)
		(r5) edge (r2)		
		(r1) edge (r5)		
		(r1) edge (r4)		
		(r2) edge (r4)
		(r3) edge (r6)
		(r4) edge (r3)	
		(r9) edge (r7)
		(r9) edge (r8);
		\draw[densely dashed]  (-.16,1.1)-- (.95,1.4) -- (.95,-.4)--(-.16,-.1)--(-.16,1.1) ;	
		\draw[densely dashed] (1.6,.5) ellipse (.15 and .7);
		\draw[densely dashed] (2.6,.5) ellipse (.15 and .7);
		\draw[densely dashed] (3.1,.5) ellipse (.15 and .28);
		\end{tikzpicture}\end{array}$
		&
		$\begin{array}{c}\begin{tikzpicture}[scale=0.7]
		\tikzstyle{vertex}=[circle, draw, inner sep=0pt, minimum size=2pt]
		\vertex[fill] (r1) at (0,0) [] {};
		\vertex[fill] (r3) at (.5,.5) [] {};
		\vertex[fill] (r2) at (.5,-.5) [] {};
		\vertex[fill] (r5) at (1.5,.5) [] {};
		\vertex[fill] (r4) at (1.5,-.5) [] {};
		\vertex[fill] (r6) at (2.3,-.5) [] {};
		\vertex[fill] (r7) at (2.3,.5) [] {};
		\vertex[fill] (r8) at (3.3,-.5) []{};
		\vertex[fill] (r9) at (3.3,.5) [] {};
		\vertex[fill] (r10) at (3.8,0) [] {};
		\path
		(r8) edge (r7)
		(r7) edge (r9)
		(r3) edge (r4)
		(r1) edge (r4)
		(r2) edge (r4)
		(r5) edge (r7)	
		(r8) edge (r10)
		(r9) edge (r10)
		(r6) edge (r8)
		(r6) edge (r9)	;
		\path
		(r1) edge (r2)
		(r1) edge (r3)
		(r2) edge (r3)	        	
		(r6) edge (r7)					 	
		(r6) edge (r4)
		(r1) edge (r5)				
		(r2) edge (r5)		
		(r3) edge (r5)	
		(r8) edge (r9)  ;	
		\draw[densely dashed]  plot[smooth, tension=1] coordinates {(-.2,0) (.6,.7)(1.63,.45) (.52,-.65) (-.2,0) };
		\draw[densely dashed]  (2.2,.65)--  (2.45,.65)--(2.45,-.65)--(1.2,-.65)--(2.2,.65) ;
		\draw[densely dashed] (3.8,0) ellipse (.15 and .28);	
		\draw[densely dashed] (3.3,0) ellipse (.15 and .7);	
		\end{tikzpicture}\end{array}$
		\\ 		\hline
	\end{tabular}
	\caption{Blocks from $D_1,\ldots,D_4$ which fit each of $D_0,\ldots,D_3$}\label{tab}
\end{table}

The next lemma captures the effect of replacing a block $D$ by another block $D'$  (which fits $D$) on the algebraic connectivity.

\begin{lemma}\label{lem:muGnEndBlock}
	Let $G$ be a graph and $D$ be an end block of $G$ attaching to the rest of the graph through a cut vertex $v$.
		Let $D$ have an equitable partition $\Pi$ such that the componenets of a Fiedler vector $\x$ of $G$ on each cell of $\Pi$  are equal.
		If we replace $D$ by another block $D'$ which fits $D$ and $v\in C_p\cap C'_p$, then for the resulting graph $G'$, we have $\mu(G')\le\mu(G)$.	
\end{lemma}
\begin{proof} Let $\Pi$ and $\Pi'$ be as in Definition~\ref{def:fit} and $d_{ij}$ be the number of edges between $C_i$ and $C_j$ for	$1\le i<j\le p$ which is the same for  $C'_i$ and $C'_j$.
	Let $\x=( x_1,\ldots, x_n)$. % be a Fiedler vector of $G$.
	By the assumption, the components of $\x$ on each $C_i$ are equal. Suppose that $a_1,\ldots,a_p$ are  the values taken by $\x$ on these components.
	It turns out that
	$$\sum_{ij\in E(D)} (x_i-x_j)^2=\sum_{1\le i<j\le p} d_{ij}(a_i-a_j)^2.$$
	Let $H=G\setminus(D-v)$.
	Now, we define a vector $\x'$ (with length $n'$) on $G'$ as follows:
	for $i=1,\ldots,p$, on each $C'_i$, all the components of $\x'$ are equal to $a_i$, and on $H$, $\x'$ is the same as $\x$.
	Although $v$ belongs to both $D'$ and $H$, the component of $\x'$ on $v$ is well-defined and equals to $a_p$ (which  is guaranteed by the assumption that $v$ belongs to both $C_p$ and $C'_p$).
	 It follows that
	$$\sum_{ij\in E(D')} (x'_i-x'_j)^2=\sum_{1\le i<j\le p} d_{ij}(a_i-a_j)^2,$$
	which in turn implies that
	$$\x' L(G')\x'^\top=\sum_{ij\in E(G')} (x'_i-x'_j)^2=\sum_{ij\in E(G)} (x_i-x_j)^2=\x L(G)\x^\top.$$
	
	Let $c_i=|C_i|$ and $c'_i=|C'_i|$. By the property (ii) of Definition~\ref{def:fit}, $c'_i\ge c_i$.
	Since $\x\1^\top=0$, it is seen that $\delta=\x'\1^\top=\sum_{i=1}^p(c'_i-c_i)a_i$.
	Moreover,
	$$\|\x'\|^2-\|\x\|^2=\sum_{i=1}^p(c'_i-c_i)a_i^2,$$
	and so
	$$\|\x'\|^2-\frac{\delta^2}{n'}-\|\x\|^2=\sum_{i=1}^p(c'_i-c_i)a_i^2-\frac1{n'}\left(\sum_{i=1}^p(c'_i-c_i)a_i\right)^2.$$
The right hand side is non-negative because by the Cauchy--Schwarz inequality,
	$$\left(\sum_{i=1}^p(c'_i-c_i)a_i\right)^2\le\left(\sum_{i=1}^p(c'_i-c_i)\right)\left(\sum_{i=1}^p(c'_i-c_i)a_i^2\right)\le n'\sum_{i=1}^p(c'_i-c_i)a_i^2.$$
It follows that $\|\x'\|^2-\frac{\delta^2}{n'}\ge\|\x\|^2$.
	Therefore, by  Lemma~\ref{remark:delta},
	$$\mu(G')\leq\frac{\x'L(G')\x'^\top}{\|\x'\|^2-\frac{\delta^2}{n'}}\le\frac{\x L(G)\x^\top}{\|\x\|^2}=\mu(G).$$
\end{proof}

\begin{remark} \label{rem:sign}\rm
	In \cite{Abdi}, it was proved that minimal quartic graphs belong to a family of graphs with path-like structure whose blocks, among others, include $M_0$ as middle blocks and  $D_0,\ldots,D_4$ as end blocks (see Theorem~\ref{thm:quarticOLD} below for the presice description of such graphs). In \cite{Abdi},
	 it was also shown that any graph of this family has an equitable partition $\Pi$ (whose cells, among others, include the  pairs of vertices  drawn vertically above each other) and a Fiedler vector $\x$ such that
(i) the components of $\x$ on each cell of $\Pi$  are equal;
	(ii)  the componenets of $\x$ on the cells of $\Pi$ form a strictly decreasing sequence changing sign once.
\end{remark}

We now present the main result of this section.

\begin{theorem} \label{thm:muGn}
	Let $G$ be  a quartic graph of order $n$ with a path-like structure
	whose middle blocks are  all $M_0$ and end blocks are from $D_0,\ldots,D_4$. Then  $\mu(G)=(1+o(1))\frac{4\pi^2}{n^2}$.
\end{theorem}
\begin{proof}
	We denote the quartic graph with $m$ middle blocks $M_0$ and end blocks $D_i$ and $D_j$ by
	$\h_{i,j}$, for $0\le i,j\le4$. The order of $\h_{i,j}$ is between $5m+11$ (the order of $\h_{0,0}$) and $5m+19$ (the order of $\h_{4,4}$).
	So it is enough to prove that $\mu(\h_{i,j})=(1+o(1))\frac{4\pi^2}{25m^2}$.
	
	We first consider the graph $\h_{0,0}$.
Define  the vector $\x=(x_1,\ldots,x_{m+1})$ with
	$$x_i=\cos\left(\frac{(2i-1)\pi}{2m+2}\right),~~i=1,\ldots,m+1.$$
	Note that  $\x$ is a skew symmetric vector,  i.e. $x_i=- x_{m+2-i}$ for $1\le i\le\lceil(m+1)/2\rceil$. We extend $\x$ to define a vector $\bar\x$ on $\h_{0,0}$ whose components are as given in Figure~\ref{fig:H00}.
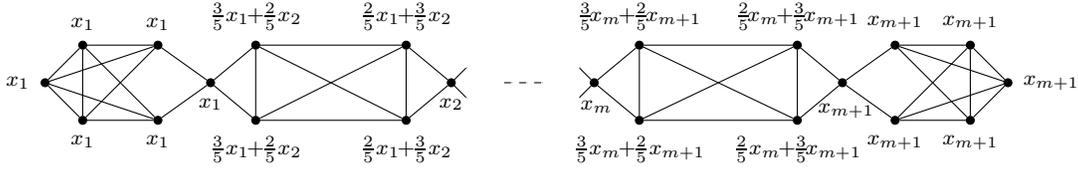
\begin{figure}[h!]
	\centering
	\begin{tikzpicture}
	\vertex[fill] (1) at (-.9,0) [label=left:\scriptsize{$x_1$}]{};
	\vertex[fill] (2) at (-.4,.5) [label=above:\scriptsize{$x_1$}] {};
	\vertex[fill] (3) at (-.4,-.5) [label=below:\scriptsize{$x_1$}] {};
	\vertex[fill] (4) at (.6,.5) [label=above:\scriptsize{$x_1$}] {};
	\vertex[fill] (5) at (.6,-.5) [label=below:\scriptsize{$x_1$}] {};
	\vertex[fill] (6) at (1.3,0) [label=below:\scriptsize{$x_1$}] {};
	\vertex[fill] (7) at (1.9,.5) [label=above:\scriptsize{$\frac35x_1\!\!+\!\!\frac25x_2$}] {};
	\vertex[fill] (8) at (1.9,-.5) [label=below:\scriptsize{$\frac35x_1\!\!+\!\!\frac25x_2$}] {};
	\vertex[fill] (9) at (3.9,.5) [label=above:\scriptsize{$\frac25x_1\!\!+\!\!\frac35x_2$}] {};
	\vertex[fill] (10) at (3.9,-.5) [label=below:\scriptsize{$\frac25x_1\!\!+\!\!\frac35x_2$}] {};
	\vertex[fill] (11) at (4.5,0) [label=below:\scriptsize{$x_2$}] {};
	\vertex[fill] (22) at (6.4,0) [] {};
	\vertex[fill] (23) at (7,.5) [label=above:\scriptsize{$\frac35\!x_m\!\!+\!\!\frac25\!x_{m+1}$}] {};
	\vertex[fill] (24) at (7,-.5) [label=below:\scriptsize{$\frac35x_m\!\!+\!\!\frac25x_{m+1}$}] {};
	\vertex[fill] (25) at (9.1,.5) [label=above:\scriptsize{$\frac25\!x_m\!\!+\!\!\frac35\!x_{m+1}$}] {};
	\vertex[fill] (26) at (9.1,-.5) [label=below:\scriptsize{$\frac25x_m\!\!+\!\!\frac35\!x_{m+1}$}] {};
	\vertex[fill] (27) at (9.7,0) [] {};
	\vertex[fill] (28) at (10.4,.5) [label=above:\scriptsize{$x_{m+1}$}] {};
	\vertex[fill] (29) at (10.4,-.5) [label=below:\scriptsize{$x_{m+1}$}] {};
	\vertex[fill] (30) at (11.4,.5) [label=above:\scriptsize{$x_{m+1}$}] {};
	\vertex[fill] (31) at (11.4,-.5) [label=below:\scriptsize{$x_{m+1}$}] {};
	\vertex[fill] (32) at (11.9,0) [label=right:\scriptsize{$x_{m+1}$}] {};
	\tikzstyle{vertex}=[circle, draw, inner sep=0pt, minimum size=0pt]
	\vertex[fill] (12) at (4.7,.2)[] {};
	\vertex[fill] (13) at (4.7,-.2)[] {};
	\vertex[fill] (14) at (5.2,0)[]{} ;
	\vertex[fill] (15) at (5.3,0)[]{} ;
	\vertex[fill] (16) at (5.4,0)[]{} ;
	\vertex[fill] (17) at (5.5,0)[]{} ;
	\vertex[fill] (18) at (5.6,0)[]{} ;
	\vertex[fill] (19) at (5.7,0)[]{} ;
	\vertex[fill] (20) at (6.2,.2)[] {};
	\vertex[fill] (21) at (6.2,-.2)[] {};
	\vertex[] () at (6.43,-.09) [label=below:\scriptsize{$x_m$}] {};
	\vertex[] () at (9.75,-.13) [label=below:\scriptsize{$x_{m+1}$}] {};
	\path
	(1) edge (2)
	(1) edge (3)
	(1) edge (4)
	(1) edge (5)
	(2) edge (3)
	(2) edge (4)
	(2) edge (5)
	(3) edge (4)
	(3) edge (5)
	(4) edge (6)
	(5) edge (6)
	(6) edge (7)
	(6) edge (8)
	(7) edge (8)
	(7) edge (9)
	(7) edge (10)
	(8) edge (9)
	(8) edge (10)
	(9) edge (10)
	(9) edge (11)
	(10) edge (11)
	(11) edge (12)
	(11) edge (13)
	(15) edge (14)
	(17) edge (16)
	(19) edge (18)
	(20) edge (22)
	(21) edge (22)
	(22) edge (23)
	(22) edge (24)
	(23) edge (24)
	(23) edge (25)
	(23) edge (26)
	(24) edge (25)
	(24) edge (26)
	(25) edge (26)
	(25) edge (27)
	(26) edge (27)
	(27) edge (28)
	(27) edge (29)
	(28) edge (30)
	(28) edge (31)
	(28) edge (32)
	(29) edge (30)
	(29) edge (31)
	(29) edge (32)
	(30) edge (31)
	(30) edge (32)
	(31) edge (32) ;
	\end{tikzpicture}
	\caption{The graph $\h_{0,0}$ and the components of $\bar\x$}
	\label{fig:H00}
	\end{figure}
	
	As $\x$ is skew symmetric, $\bar\x$ is also skew symmetric. It follows that $\bar\x\perp\bf1$. Therefore, by \eqref{eq2} we have
	\begin{align}
	\frac{\bar\x L(\h_{0,0})\bar\x^\top}{\bar\x\bar\x^\top}\nonumber
	&= \frac{\frac{20}{25} \sum_{i=1}^{m}(x_i-x_{i+1})^2}{\sum_{i=1}^{m+1}x_i^2+\frac{2}{25}\sum_{i=1}^{m}(3x_{i}+2x_{i+1})^2+\frac{2}{25}\sum_{i=1}^{m}(2x_{i}+3x_{i+1})^2+10x_1^2}\nonumber \\
	&=  \frac{\frac{20}{25} \sum_{i=1}^{m}(x_i-x_{i+1})^2}{\sum_{i=1}^{m+1}x_i^2+\frac{2}{25}\sum_{i=1}^{m}(13x^2_i+24x_ix_{i+1}+13x^2_{i+1})+10x_1^2}\nonumber \\
	&\leq  \frac{\frac{20}{25} \sum_{i=1}^{m}(x_i-x_{i+1})^2}{\frac{77}{25}\sum_{i=1}^{m}x_i^2+\frac{48}{25}\sum_{i=1}^{m}x_ix_{i+1}}\nonumber \\
	&= \frac{80\sin^2(\frac{\pi}{2m+2})\sum_{i=1}^{m}\sin^2(\frac{\pi  i}{m+1})}{77\sum_{i=1}^{m}\cos^2(\frac{(2i-1)\pi}{2m+2})
		+24(m\cos\frac{\pi}{m+1}+\sum_{i=1}^m\cos(\frac{2\pi i}{m+1}))} \label{eq:cos-sin}\\
	&=\frac{40(m+1) \sin^2(\frac{\pi}{2m+2})}{77\left(\frac{m+1}{2}-\cos^2(\frac{\pi}{2m+2})\right)+24(m\cos(\frac{\pi}{m+1})-1)}\label{eq:sin^2} \\
	&=(1+o(1))\frac{4\pi^2}{25m^2}. \nonumber
	\end{align}
 The last equality is obtained using Taylor's series for sine and cosine. Note that \eqref{eq:cos-sin} is obtained using the identities
	\begin{align*}
	\cos\alpha-\cos\beta&=2 \sin \frac{\alpha+\beta}{2} \sin \frac{\beta-\alpha}{2},\\
	\cos\alpha \cos\beta&=\frac{1}{2}( \cos(\alpha-\beta)+\cos(\alpha+\beta)),
	\end{align*}
	and   \eqref{eq:sin^2} is deduced from the identities
	\begin{align*}
	\sum_{i=1}^{m}\sin^2\left(\frac{\pi i}{m+1}\right)&=\frac{m+1}{2},\\
	\sum_{i=1}^m\cos\left(\frac{2\pi i}{m+1}\right)&=-1,\\
	\sum_{i=1}^{m}\cos^2\left(\frac{(2i-1)\pi}{2m+2}\right)&=\frac{m+1}{2}-\cos^2\left(\frac{\pi}{2m+2}\right).
	\end{align*}
	Therefore, we conclude that
	\begin{equation}\label{eq:g0,0}
	\mu(\h_{0,0})\leq (1+o(1))\frac{4\pi^2}{25m^2}.
	\end{equation}

	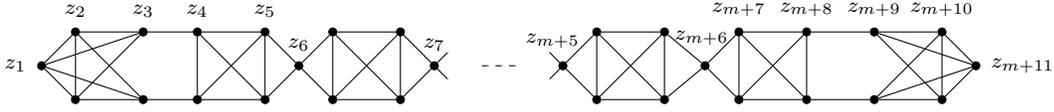
\begin{figure}
		\centering
\begin{tikzpicture}[scale=0.9]
	\vertex[fill] (1) at (-1.8,0) [label=left:\scriptsize{$z_1$}]{};
	\vertex[fill] (2) at (-1.3,.5) [label=above:\scriptsize{$z_2$}] {};
	\vertex[fill] (3) at (-1.3,-.5) [] {};
	\vertex[fill] (4) at (-.3,.5) [label=above:\scriptsize{$z_3$}] {};
	\vertex[fill] (5) at (-.3,-.5) [] {};
	\vertex[fill] (06) at (.5,.5) [label=above:\scriptsize{$z_4$}] {};
	\vertex[fill] (07) at (.5,-.5) [] {};
	\vertex[fill] (08) at (1.5,.5) [label=above:\scriptsize{$z_5$}] {};
	\vertex[fill] (09) at (1.5,-.5) [] {};
	\vertex[fill] (6) at (2,0) [label=above:\scriptsize{$z_6$}] {};
	\vertex[fill] (7) at (2.5,.5) [] {};
	\vertex[fill] (8) at (2.5,-.5) [] {};
	\vertex[fill] (9) at (3.5,.5) [] {};
	\vertex[fill] (10) at (3.5,-.5) [] {};
	\vertex[fill] (11) at (4,0) [label=above:\scriptsize{$z_7$}] {};
	\vertex[fill] (22) at (5.9,0) [] {};
	\vertex[fill] (23) at (6.4,.5) [] {};
	\vertex[fill] (24) at (6.4,-.5) [] {};
	\vertex[fill] (25) at (7.4,.5) [] {};
	\vertex[fill] (26) at (7.4,-.5) [] {};
	\vertex[fill] (27) at (8,0) [] {};
	\vertex[fill] (028) at (8.5,.5) [label=above:\scriptsize{$z_{m+7}$}] {};
	\vertex[fill] (029) at (8.5,-.5) [] {};
	\vertex[fill] (030) at (9.5,.5) [label=above:\scriptsize{$z_{m+8}$}] {};
	\vertex[fill] (031) at (9.5,-.5) [] {};
	\vertex[fill] (28) at (10.5,.5) [label=above:\scriptsize{$z_{m+9}$}] {};
	\vertex[fill] (29) at (10.5,-.5) [] {};
	\vertex[fill] (30) at (11.5,.5) [label=above:\scriptsize{$z_{m+10}$}] {};
	\vertex[fill] (31) at (11.5,-.5) [] {};
	\vertex[fill] (32) at (12,0) [label=right:\scriptsize{$z_{m+11}$}] {};
	\tikzstyle{vertex}=[circle, draw, inner sep=0pt, minimum size=0pt]
	\vertex[fill] (12) at (4.2,.2)[] {};
	\vertex[fill] (13) at (4.2,-.2)[] {};
	\vertex[fill] (14) at (4.7,0)[]{} ;
	\vertex[fill] (15) at (4.8,0)[]{} ;
	\vertex[fill] (16) at (4.9,0)[]{} ;
	\vertex[fill] (17) at (5,0)[]{} ;
	\vertex[fill] (18) at (5.1,0)[]{} ;
	\vertex[fill] (19) at (5.2,0)[]{} ;
	\vertex[fill] (20) at (5.7,.2)[] {};
	\vertex[fill] (21) at (5.7,-.2)[] {};
	\vertex[] () at (5.75,.1) [label=above:\scriptsize{$z_{m+5}$}] {};
	\vertex[] () at (7.95,.14) [label=above:\scriptsize{$z_{m+6}$}] {};
	\path
	(1) edge (2)
	(1) edge (3)
	(1) edge (4)
	(1) edge (5)
	(2) edge (3)
	(2) edge (4)
	(2) edge (5)
	(3) edge (4)
	(3) edge (5)
	(4) edge (06)
	(5) edge (07)
	(08) edge (06)
	(07) edge (06)
	(09) edge (08)
	(09) edge (06)
	(08) edge (07)
	(09) edge (07)
	(6) edge (08)
	(6) edge (09)
	(7) edge (6)
	(6) edge (8)
	(7) edge (8)
	(7) edge (9)
	(7) edge (10)
	(8) edge (9)
	(8) edge (10)
	(9) edge (10)
	(9) edge (11)
	(10) edge (11)
	(11) edge (12)
	(11) edge (13)
	(15) edge (14)
	(17) edge (16)
	(19) edge (18)
	(20) edge (22)
	(21) edge (22)
	(22) edge (23)
	(22) edge (24)
	(23) edge (24)
	(23) edge (25)
	(23) edge (26)
	(24) edge (25)
	(24) edge (26)
	(25) edge (26)
	(25) edge (27)
	(26) edge (27)
	(27) edge (028)
	(27) edge (029)
	(029) edge (028)
	(028) edge (030)
	(028) edge (031)
	(029) edge (030)
	(029) edge (031)
	(030) edge (031)
	(030) edge (28)
	(29) edge (031)
	(28) edge (30)
	(28) edge (31)
	(28) edge (32)
	(29) edge (30)
	(29) edge (31)
	(29) edge (32)
	(30) edge (31)
	(30) edge (32)
	(31) edge (32) ;
	\end{tikzpicture}
	\caption{The graph $\h_{4,4}$ and the components of $\z$}
	\label{fig:H44}
\end{figure}
		We now prove that $(1+o(1))\frac{4\pi^2}{25m^2}$ is a lower bound for $\mu(\h_{4,4})$.
	Let $\y=(y_1,\ldots,y_n)$ be a Fiedler vector of $\h_{4,4}$.
	The graph $\h_{4,4}$  has $m+1$ cut vertices and $m+2$ blocks, say $B_1,\ldots,B_{m+2}$.
	Consider the components of $\y$ on the cut vertices and on the end blocks of $\h_{4,4}$
	which give rise to a vector $\z$ consisting of $m+11$ components as depicted in Figure~\ref{fig:H44}.
	Note  that $\y$ is skew symmetric.
	To verify this, observe that by the symmetry of $\h_{4,4}$, $\y'=(y_n,y_{n-1},\ldots,y_1)$ is also an eigenvector for $\mu(\h_{4,4})$.
	It follows that $\y-\y'$ itself is a skew symmetric eigenvector for $\mu(\h_{4,4})$ (note that from Remark~\ref{rem:sign}, it is seen that $\y-\y'\ne\bf0$), so that we may replace $\y-\y'$ for $\y$. Now, from Remark~\ref{rem:sign} it follows that $\z=(z_1, z_2, \ldots, z_{m+11})\ne\bf0$.
	As $\y$ is skew symmetric, it follows that $\z$ is also skew symmetric and thus $\z\perp\bf1$.
%Let $6\le k\le m+5$ and
Let $B_r$ be one of the middle blocks of $\h_{4,4}$ and the components of $\y$ on the left vertex and the right vertex of $B_r$  be $z_{k}$ and $z_{k+1}$, respectively.
	Let $s$ and
	$t$ be the components of $\y$ on the two middle vertices of $B_r$ (which are equal by Remark~\ref{rem:sign}) as shown in Figure \ref{fig:Bk}.
\begin{figure}[h!]
	\centering\begin{tikzpicture}[scale=0.8]
	\vertex[fill] (r) at (0,0) [label=left:\footnotesize{$z_k$}] {};{};
	\vertex[fill] (r1) at (.5,.5) [label=above:\footnotesize{$s$}] {}; {};
	\vertex[fill] (r2) at (.5,-.5) [label=below:\footnotesize{$s$}] {};
	\vertex[fill] (r3) at (1.5,.5) [label=above:\footnotesize{$t$}] {}; {};
	\vertex[fill] (r4) at (1.5,-.5) [label=below:\footnotesize{$t$}] {};
	\vertex[fill] (r5) at (2,0) [label=right:\footnotesize{$z_{k+1}$}] {}; {};
	\tikzstyle{vertex}=[circle, draw, inner sep=0pt, minimum size=0pt]
	\vertex (s) at (-.2,.2) []{};
	\vertex (ss) at (-.2,-.2) []{};
	\vertex (sss) at (2.2,.2) [] {};
	\vertex (ssss) at (2.2,-.2) [] {};
	\path
	(r) edge (r1)
	(r) edge (r2)
	(r1) edge (r2)
	(r1) edge (r3)
	(r1) edge (r4)
	(r2) edge (r3)
	(r2) edge (r4)
	(r3) edge (r4)
	(r5) edge (r4)
	(r3) edge (r5)
	(r) edge (s)
	(r) edge (ss)
	(r5) edge (sss)
	(r5) edge (ssss);
	\end{tikzpicture}
	\caption{The middle  $B_r$ and the components of $\y$ on its vertices}
\label{fig:Bk}		
	\end{figure}
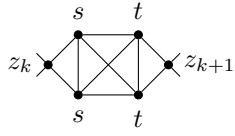

We have
$$\sum_{ij\in E(B_r)}(y_i-y_j)^2=2(z_k-s)^2+4(s-t)^2+2(t-z_{k+1})^2.$$
The right hand side, considered as a function of $s$ and $t$,  is minimized at $s=\frac15(3z_{k}+2z_{k+1})$ and
$t=\frac{1}{5}(2z_k+3z_{k+1})$. This implies that
$$\sum_{ij\in E(B_r)}(y_i-y_j)^2\ge\frac{4}{5}(z_{k}-z_{k+1})^2.$$
It follows that
\begin{align}\label{u}
\sum_{ij\in E(\h_{4,4})}(y_i-y_j)^2 &= \sum_{ij\in E(B_1)}(y_i-y_j)^2+\sum_{r=2}^{m+1}\sum_{ij\in E(B_r)}(y_i-y_j)^2 +\sum_{ij\in E(B_{m+2})}(y_i-y_j)^2 \nonumber \\
&\geq \frac{4}{5}\sum_{k=1}^{m+10} (z_{k}-z_{k+1})^2.
\end{align}
Since $\y$ is skew-symmetric and its components are decreasing (by Remark~\ref{rem:sign}), for the middle block $B_r$ we have $z_k^2\ge s^2\ge t^2$ if $2\le r\le (m+3)/2$, and $s^2\le t^2\le z_{k+1}^2$ if $(m+3)/2\le r\le m$.
Furthermore, $\sum_{i=1}^9y_i^2 \leq 2\sum_{i=1}^{5}z_i^2$ and
$\sum_{i=n-8}^ny_i^2 \leq 2\sum_{i=m+7}^{m+11}z_i^2$.
It turns out that
\begin{equation} \label{d}
\sum_{i=1}^ny_i^2 \leq 5\sum_{i=1}^{m+11}z_i^2.
\end{equation}
Now, from  \eqref{u} and \eqref{d}, we  infer that
\begin{equation} \label{eq:P_{m+11}}
\mu(\h_{4,4})=\frac{\y L(\h_{4,4})\y^\top}{\y\y^\top}\geq \frac{4\sum_{i=1}^{m+10}(z_i-z_{i+1})^2}{25\sum_{i=1}^{m+11}z_i^2}.
\end{equation}
Note that the right hand side of \eqref{eq:P_{m+11}} is the same as
$\frac4{25}\frac{\z L(P_{m+11})\z^\top}{\z\z^\top}$, where $P_{m+11}$ is the path of order $m+11$.
Thus, by the fact that  $\mu(P_h)=2(1-\cos\frac{\pi}{h})$ (see \cite{fiedler1973algebraic}),
it follows that
\[\frac{\sum_{i=1}^{m+10}(z_i-z_{i+1})^2}{\sum_{i=1}^{m+11}z_i^2}\ge\mu(P_{m+11})=(1+o(1))\frac{\pi^2}{m^2}.\]
Therefore, by \eqref{eq:P_{m+11}}
\begin{equation} \label{eq:muH4,4}
\mu(\h_{4,4})\geq \frac{4}{25}\mu(P_{m+11}) = (1+o(1))\frac{4\pi^2}{25m^2}.
\end{equation}

	From Table~\ref{tab} it is evident that all $D_i$'s fit $D_0$. Also, by Remark \ref{rem:sign}, they fulfill the conditions of Lemma~\ref{lem:muGnEndBlock}. So by  applying twice Lemma~\ref{lem:muGnEndBlock}, we obtain
	$$\mu(\h_{i,j})\le\mu(\h_{0,j})\le\mu(\h_{0,0}).$$
	Also $D_4$ fits $D_3$ and $D_3$ fits each of $D_0,D_1$, and $D_2$. Therefore, again by Lemma~\ref{lem:muGnEndBlock},
	$$\mu(\h_{i,j})\ge\mu(\h_{4,j})\ge\mu(\h_{4,4}).$$
	The result now follows from \eqref{eq:g0,0} and \eqref{eq:muH4,4}.
\end{proof}

\section{Structure of minimal quartic graphs}  \label{sec:structure}

 Motivated by  the Aldous--Fill conjecture and also as an analogue  to Babai's conjecture on minimal cubic graphs,
 we consider the problem of determining the structure of minimal quartic graphs.
 In \cite{Abdi}, it was proved that minimal quartic graphs have a path-like structure  with specified blocks. We start this section by quoting this result.

The possible blocks of minimal quartic graphs are
of two types: `short' and `long'.
By {\em short} blocks we mean the blocks $M_0,D_0,D_1,D_2,D_3$ and those given in Figure~\ref{fig:MiDi}.
The {\em long} blocks, roughly speaking,  are constructed by putting some short blocks together with the general structure given in Figure~\ref{fig:long}.
More precisely, the building `bricks' of long blocks  are the graphs $M'_0,M'_1,M'_2,D'_0,D'_3$, obtained by removing the {\em right} degree $2$ vertex of
the corresponding short blocks, as well as the graphs $M''_0,M''_1$,  obtained by removing both degree $2$ vertices of $M_0,M_1$.
For any of these graphs, say $B$, we denote its mirror image by $\tilde B$.

 \begin{figure}[h!]
 	\captionsetup[subfigure]{labelformat=empty}
 	\centering
\subfloat[$M_1$]{\begin{tikzpicture}[scale=.9]
\vertex[fill] (r) at (0,0) [] {};
	\vertex[fill] (r1) at (.5,.5) [] {};
	\vertex[fill] (r2) at (.5,-.5) [] {};
	\vertex[fill] (r3) at (1.5,.5) [] {};
	\vertex[fill] (r4) at (1.5,-.5) [] {};
	\vertex[fill] (r5) at (2.5,.5) [] {};
	\vertex[fill] (r6) at (2.5,-.5) [] {};
\vertex[fill] (r7) at (3,0) [] {};
	\path
    ( r) edge (r1)
	(r) edge (r2)
	(r5) edge (r6)
	(r1) edge (r2)
	(r1) edge (r3)
	(r1) edge (r4)
	(r2) edge (r3)
	(r2) edge (r4)
	(r3) edge (r6)
	(r4) edge (r5)
	(r5) edge (r3)
	(r4) edge (r6)
    (r5) edge (r7)
	(r7) edge (r6);
	\end{tikzpicture}}
\qquad
\subfloat[$M_2$]{\begin{tikzpicture}[scale=.9]
    \vertex[fill] (r) at (1,0) [] {};
	\vertex[fill] (r1) at (1.5,.5) [] {};
	\vertex[fill] (r2) at (1.5,-.5) [] {};
	\vertex[fill] (r3) at (2,0) [] {};
	\vertex[fill] (r4) at (2.5,.5) [] {};
	\vertex[fill] (r5) at (2.5,-.5) [] {};
	\vertex[fill] (r6) at (3.5,.5) [] {};
	\vertex[fill] (r7) at (3.5,-.5) [] {};
   \vertex[fill] (r8) at (4,0) [] {};
	\path
    (r) edge (r2)
	(r1) edge (r)
	(r1) edge (r2)
	(r1) edge (r3)
	(r1) edge (r4)
	(r2) edge (r3)
	(r2) edge (r5)
	(r3) edge (r4)
	(r3) edge (r5)
	(r6) edge (r4)
	(r5) edge (r7)
	(r5) edge (r6)
	(r4) edge (r7)
	(r6) edge (r7)
    (r6) edge (r8)
	(r7) edge (r8) ;
	\end{tikzpicture}}
 \qquad
\subfloat[$M_3$]{\begin{tikzpicture}[scale=.9]
       \vertex[fill] (r) at (1,0) [] {};
 		\vertex[fill] (r1) at (1.5,.5) [] {};
 		\vertex[fill] (r2) at (1.5,-.5)[] {};
 		\vertex[fill] (r3) at (2,0) [] {};
 		\vertex[fill] (r4) at (2.5,.5) [] {};
 		\vertex[fill] (r5) at (2.5,-.5) [] {};
 	\vertex[fill] (r6) at (3,0) [] {};
 		\path
 	   (r) edge (r1)
 		(r) edge (r2)
 		(r1) edge (r2)
 		(r1) edge (r3)
 		(r1) edge (r4)
 		(r2) edge (r3)
 		(r2) edge (r5)
 		(r3) edge (r4)
 		(r3) edge (r5)
 		(r5) edge (r4)
 	     (r6) edge (r5)
 		(r6) edge (r4) 	;
 		\end{tikzpicture}}
 	\caption{Some of the short blocks of potentially minimal quartic graphs}\label{fig:MiDi}
 \end{figure}
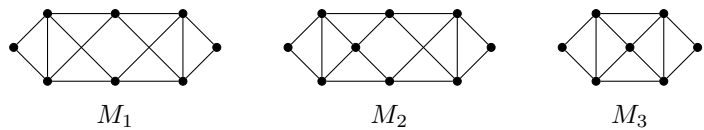

A long block is constructed from some $s\ge2$ bricks $B_1, \ldots, B_s$, where each $B_i$ is joined by two edges to $B_{i+1}$ (as shown in Figure~\ref{fig:long}) and $B_2, \ldots, B_{s-1}\in \{ M''_0, M''_1\}$.
There are three types of long blocks:
\begin{itemize}
\item[(i)] {\em long end block}: if $B_1\in\{D'_0, D'_3\}$ 	and $B_s \in \{\tilde M'_0, \tilde M'_1, \tilde M'_2\}$;
\item[(ii)] {\em  long middle block}: if  $B_1 \in\{M'_0, M'_1, M'_2\}$ and $B_s \in\{\tilde M'_0, \tilde M'_1, \tilde M'_2\}$;
\item[(iii)] {\em  long complete block}: if $B_1\in\{D'_0, D'_3\}$ and $B_s\in\{\tilde D'_0, \tilde D'_3\}$.
 \end{itemize}
 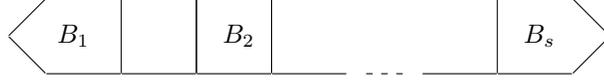
\begin{figure}[h!]
 	\centering
 	\begin{tikzpicture}
 	\tikzstyle{vertex}=[draw, inner sep=0pt, minimum size=0pt]
 	\vertex[fill] (r17) at (9,0) [label=left:\footnotesize{$$}] {};
 	\vertex[fill] (r) at (1,0) [label=left:\footnotesize{$$}] {};
 	\vertex[fill] (r1) at (1.5,.5) [] {};
 	\vertex[fill] (r2) at (1.5,-.5) [] {};
 	\vertex[fill] (r3) at (2.5,.5) [] {};
 	\vertex[fill] (r4) at (2.5,-.5) [] {};
 	\vertex[fill] (r5) at (3.5,.5) [] {};
 	\vertex[fill] (r6) at (3.5,-.5) [] {};
 	\vertex[fill] (r7) at (4.5,.5) [] {};
 	\vertex[fill] (r8) at (4.5,-.5) [] {};
 	\vertex[fill] (r9) at (5.5,.5) [] {};
 	\vertex[fill] (r10) at (5.5,-.5) [] {};
 	\vertex[fill] (r11) at (6.5,.5) [] {};
 	\vertex[fill] (r12) at (6.5,-.5) [] {};
 	\vertex[fill] (r13) at (7.5,.5) [] {};
 	\vertex[fill] (r14) at (7.5,-.5) [] {};
 	\vertex[fill] (r15) at (8.5,.5) [] {};
 	\vertex[fill] (r16) at (8.5,-.5) [] {};
 	\vertex[fill] (b1) at (1.5,0) [label=right:\footnotesize{$B_1$}] {};
 	\vertex[fill] (b1) at (3.7,0) [label=right:\footnotesize{$B_2$}] {};
 	\vertex[fill] (b1) at (7.7,0) [label=right:\footnotesize{$B_s$}] {};
 	\tikzstyle{vertex}=[ draw, inner sep=0pt, minimum size=0pt]
 	\vertex[] (1) at (5.75,.5) [] {};
 	\vertex[] (2) at (5.95,.5) [] {};
 	\vertex[] (3) at (6.15,.5) [] {};
 	\vertex[] (11) at (5.85,.5) [] {};
 	\vertex[] (22) at (6.05,.5) [] {};
 	\vertex[] (33) at (6.25,.5) [] {};
 	\vertex[] (111) at (5.75,-.5) [] {};
 	\vertex[] (222) at (5.95,-.5) [] {};
 	\vertex[] (333) at (6.15,-.5) [] {};
 	\vertex[] (1111) at (5.85,-.5) [] {};
 	\vertex[] (2222) at (6.05,-.5) [] {};
 	\vertex[] (3333) at (6.25,-.5) [] {};
 	\path
 	(r) edge (r1)
 	(r) edge (r2)
 	(r2) edge (r4)
 	(r1) edge (r3)
 	(r3) edge (r4)
 	(r3) edge (r5)
 	(r4) edge (r6)
 	(r5) edge (r6)
 	(r5) edge (r6)
 	(r7) edge (r8)
 	(r5) edge (r7)
 	(r6) edge (r8)
 	(r7) edge (r9)
 	(r8) edge (r10)
 	(r11) edge (r13)
 	(r12) edge (r14)
 	(r13) edge (r14)
 	(r13) edge (r15)
 	(r14) edge (r16)
 	(r16) edge (r17)
 	(r15) edge (r17)
 	(1) edge (11)
 	(2) edge (22)
 	(3) edge (33)
 	(111) edge (1111)
 	(222) edge (2222)
 	(333) edge (3333);
 	\end{tikzpicture}
  \caption{General structure of a long block}\label{fig:long}
 \end{figure}

 In passing, we remark that in a quartic graph, any cut vertex belongs to exactly two blocks and moreover has degree $2$ in each of them.
 Therefore, in the quartic graphs having a path-like structure, the middle and end blocks have exactly two and one vertices of degree $2$, respectively.

 \begin{theorem}[Abdi, Ghorbani, and Imrich \cite{Abdi}] \label{thm:quarticOLD}
Let $G$ be a graph with the minimum spectral gap in the family of connected quartic graphs on $n\ge11$ vertices.
If $G$ is a block, then it is a long complete block.
 If  $G$ itself is not a block, then it has a path-like structure in which each left end  block is either one of $D_0,\ldots,D_3$ or a long end block, and
 each middle block is either one of $M_0,M_1,M_2,\tilde M_2,M_3$ or a long middle block.
  Each right end block is the mirror image of some left end block.
\end{theorem}

 As the main result of this section, we improve considerably Theorem~\ref{thm:quarticOLD} by giving a much more precise description of the minimal quartic graphs.

\begin{theorem}\label{thm:quarticNEW}
	Let $G$ be a graph with minimum algebraic connectivity among connected quartic graphs on $n\ge11$ vertices. Then $G$ has a path-like structure, any middle block of $G$ is $M_0$,  the left end block of $G$ is one of  $D_0,\ldots,D_4$, and the right end block is the mirror image of one of these five blocks.
\end{theorem}

Theorems~\ref{thm:muGn} and \ref{thm:quarticNEW} imply the Aldous--Fill conjecture for $k=4$:

\begin{corollary}\label{corollary}
The minimum algebraic connectivity of connected quartic graphs of order $n$ is $(1+o(1))\frac{4\pi^2}{n^2}$.
\end{corollary}

The rest of this section is devoted to the proof of Theorem~\ref{thm:quarticNEW}.
More precisely, the assertion on the end blocks will be established in Theorem~\ref{thm:endblocks} and on
the middle block  in Theorem~\ref{thm:middleblock} below.

In the remainder of this section,  we suppose that $n\geq 11$,
 $\G$ is a minimal quartic graph of order $n$ and $\mu:=\mu(\G)$.  By Theorem~\ref{thm:quarticOLD}, the equitable partition of $\G$ mentioned in Remark~\ref{rem:sign} has cells of size $1$ or $2$ consisting of  the vertices drawn vertically above each other (with the exceptions for the first three vertices of $D'_0$ and the first four vertices of $D_1, D'_3$ which make  a cell too). So by Remark~\ref{rem:sign}, $\G$ has a {\em decreasing} unit Fiedler vector $\x$  which is {\em constant} on the cells.

In our arguments in this section, we will need upper bounds on $\mu$.

 \begin{lemma} \label{lem:muUpBound}
 Based on the value of $n$, we have the upper bounds given in Table~\ref{tab:muUpBound} for $\mu$.
\end{lemma}
\begin{table}[ht]
\centering
\begin{tabular}{cccccc}
\hline
$n\ge$ & $11$ & $13$ & $18$ & $21$ & $26$  \\ \hline
$\mu<$ & $0.355$&$0.268$&$0.129$&$0.091$&$0.059$\\
\hline
\end{tabular}
\caption{Upper bounds for $\mu$ based on the values of $n$}
\label{tab:muUpBound}
\end{table}
\begin{proof}
Note that $\mu\le\mu(\g_n)$. So it suffices to find upper bound for $\mu(\g_n)$.
 Let $n-11=5m+r$ for some $0\le r\le4$, and $\h_{0,0}$ (with $m$ middle blocks) be as in the proof of Theorem~\ref{thm:muGn}.
Therefore, $\mu(\g_n)\le\mu(\h_{0,0})$.
Let us denote the expression given in \eqref{eq:sin^2} by $f(m)$ for which $\mu(\h_{0,0})\le f(m)$. By taking the derivative of $f$ with respect to $m$ it is seen that $f$ is decreasing for $m\ge1$. So for $m\ge6$, we have $f(m)\le f(6)<0.046$. This implies that for $n\ge41$, $\mu(\g_n)<0.046$.
 For $11\le n\le40$, we compute $\mu(\g_n)$ directly by a computer.
 It turns out that $\mu(\g_n)$ is strictly decreasing for $11\le n\le40$. The rounded up values of $\mu(\g_n)$   for
$n=11,13,18,21,26$  are exactly those given in the second row of  Table~\ref{tab:muUpBound}. Thus the result follows.
\end{proof}
\

\subsection{End blocks}

By Theorem~\ref{thm:quarticOLD}, the left end block of $\G$ is either a short block, i.e. it is one of $D_0,\ldots, D_3$, or it is a long block starting with  $D'_0$ or $D'_3$. In this subsection, we show that the end blocks of $\G$ should be short blocks or one exceptional long block that is $D_4$.

Our first lemma concerns long blocks starting with $D'_3$.
\begin{lemma}\label{lem:E3}
The graph $\G$ does not contain a long end block starting with $D'_3$.
\end{lemma}
\begin{proof}
For a contradiction suppose that $\G$ contains $D'_3$ with the components of $\x$ as depicted in Figure~\ref{fig:longD3}.
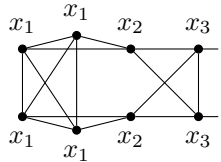
\begin{figure}[H]
\centering
\begin{tikzpicture}[scale=0.9]
	\vertex[fill] (r1) at (0,1) [label=above:\footnotesize{$x_1$}] {};
	\vertex[fill] (r2) at (.8,1.2) [label=above:\footnotesize{$x_1$}] {};
	\vertex[fill] (r3) at (0,0) [label=below:\footnotesize{$x_1$}]{};
	\vertex[fill] (r4) at (.8,-.2) [label=below:\footnotesize{$x_1$}]{};
    \vertex[fill] (r5) at (1.6,1) [label=above:\footnotesize{$x_2$}] {};
    \vertex[fill] (r6) at (1.6,0) [label=below:\footnotesize{$x_2$}]{};
    \vertex[fill] (r7) at (2.6,1) [label=above:\footnotesize{$x_3$}] {};
    \vertex[fill] (r8) at (2.6,0) [label=below:\footnotesize{$x_3$}]{};
    \tikzstyle{vertex}=[circle, draw, inner sep=0pt, minimum size=0pt]
    \vertex (s) at (2.9,1)[] {};
	\vertex (ss) at (2.9,0) [] {};
	\path
		(r5) edge (r2)
		(r1) edge (r2)
	    (r1) edge (r5)
		(r1) edge (r3)
        (r1) edge (r4)
		(r2) edge (r3)
	    (r2) edge (r4)
	    (r3) edge (r6)
	    (r4) edge (r3)
	    (r4) edge (r6)
	    (r5) edge (r8)
	    (r5) edge (r7)
	    (r6) edge (r8)
	    (r6) edge (r7)
	    (r7) edge (r8)
	   (r7) edge (s)
	   (r8) edge (ss);
\end{tikzpicture}
\caption{$D'_3$ in a long block (the last two vertices have different neighbors on their right)}\label{fig:longD3}
\end{figure}

We may assume that $x_1>0$ (since otherwise we consider $-\x$).
By  using the eigen-equation \eqref{eq:eigenvector}, we have
$$x_3=\left(\mu^2-5\mu+2\right)\frac{x_1}2.$$
As $\mu<0.355$ (by Lemma~\ref{lem:muUpBound}), we have $\mu^2-5\mu+2>0$ which implies that $x_3>0$.

Now, we replace $D'_3$ by $D_2$ to obtain $\G'$.
We define a vector $\x'$ on $V(\G')$ such that its components on the new end block $D_2$ are as given in Figure~\ref{fig:shortD2}, and on the rest of vertices agree with $\x$.
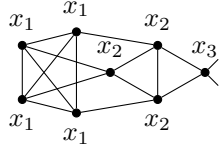
\begin{figure}[h!]
\centering
\begin{tikzpicture}[scale=0.9]
	\vertex[fill] (r1) at (0,0) [label=below:\footnotesize{$x_1$}] {};
	\vertex[fill] (r2) at (0,.8) [label=above:\footnotesize{$x_1$}] {};
	\vertex[fill] (r3) at (.8,1) [label=above:\footnotesize{$x_1$}] {};
	\vertex[fill] (r4) at (.8,-.2) [label=below:\footnotesize{$x_1$}] {};
    \vertex[fill] (r5) at (1.3,.4) [label=above:\footnotesize{$x_2$}] {};
    \vertex[fill] (r6) at (2,.8) [label=above:\footnotesize{$x_2$}] {};
    \vertex[fill] (r7) at (2,0) [label=below:\footnotesize{$x_2$}] {};
    \vertex[fill] (r8) at (2.7,.4) [label=above:\footnotesize{$x_3$}] {};
\tikzstyle{vertex}=[circle, draw, inner sep=0pt, minimum size=0pt]
      \vertex (s) at (2.9,.6)[] {};
	\vertex (ss) at (2.9,.2) [] {};
	\path
		(r5) edge (r2)
		(r1) edge (r2)
	    (r1) edge (r5)
		(r1) edge (r3)
        (r1) edge (r4)
		(r2) edge (r3)
	    (r2) edge (r4)
	    (r3) edge (r6)
	    (r4) edge (r3)
	    (r4) edge (r7)
	     (r5) edge (r7)
	    (r5) edge (r6)
	     (s) edge (r8)
	    (ss) edge (r8)
	    (r6) edge (r7)
	     (r7) edge (r8)
	      (r6) edge (r8)  ;
\end{tikzpicture}
\caption{The block $D_2$ in $\G'$ and the components of $\x'$}
\label{fig:shortD2}
\end{figure}
   We  observe that $\x'\1^\top-\x\1^\top=x_2-x_3$ and since $\x\perp\1$, $\delta=\x'\1^\top=x_2-x_3$.
     Also $\x'L(\G')\x'^\top=\x L(\G)\x^\top-2(x_2-x_3)^2=\mu-2(x_2-x_3)^2$, and
     \begin{align*}
\|\x'\|^2-\frac{\delta^2}{n}&=\|\x\|^2+x_2^2-x_3^2-\frac{\delta^2}{n}\\
 &=1+(x_2-x_3)\left(x_2+x_3-\frac{x_2-x_3}{n}\right) \\
 &\geq 1+(x_2-x_3)^2 \left(1-\frac{1}{n}\right),
\end{align*}
where the last inequality follows from the fact that $x_3>0$.
Therefore,  we have
$\x' L(\G')\x'^\top<\mu$ and
$\|\x'\|^2-\frac{\delta^2}{n}>1$. So by  Lemma~\ref{remark:delta}, $\mu(\G')\leq \frac{\x' L(\G')\x'^\top}{\|\x'\|^2-\frac{\delta^2}{n}}<\mu$, a contradiction.
\end{proof}

We now deal with the long blocks starting with $D'_0$ in the next two lemmas.
Suppose that $\G$ contains a long end block starting with $D'_0$.
Then, by Theorem~\ref{thm:quarticOLD},  it contains either $D'_0+M''_0$, $D'_0+M''_1$, or $D'_0+\tilde M'_2$ (these are the graphs obtained by joining the two degree 3 vertices of $D'_0$ by two parallel edges to $M''_0$, $M''_1$, or $\tilde M'_2$, respectively; see Figure~\ref{fig:longD4-1}).
If $\G$ contains $D'_0+M''_0$, it is either $D_4$ (as desired) or it contains the subgraph of Figure~\ref{fig:longD4-1} that we will show is not possible in Lemma~\ref{lem:E1}.
If $\G$ contains $D'_0+M''_1$ or $D'_0+\tilde M'_2$, then it must contain the subgraph of Figure~\ref{fig:longD4-2} that we shall show is not possible in Lemma~\ref{lem:E2}.

 \begin{lemma}\label{lem:E1}
 The graph  $\G$ does not contain the graph $H$ of Figure~\ref{fig:longD4-1} as a subgraph.
 \end{lemma}
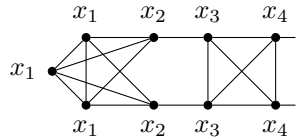
\begin{figure}[H]
\centering
\begin{tikzpicture}[scale=0.9]
	\vertex[fill] (r1) at (0,0) [label=left:\footnotesize{$x_1$}] {};
	\vertex[fill] (r3) at (.5,.5) [label=above:\footnotesize{$x_1$}] {};
	\vertex[fill] (r2) at (.5,-.5) [label=below:\footnotesize{$x_1$}] {};
	\vertex[fill] (r5) at (1.5,.5) [label=above:\footnotesize{$x_2$}] {};
	\vertex[fill] (r4) at (1.5,-.5) [label=below:\footnotesize{$x_2$}] {};
\vertex[fill] (r6) at (2.3,-.5) [label=below:\footnotesize{$x_3$}] {};
\vertex[fill] (r7) at (2.3,.5) [label=above:\footnotesize{$x_3$}] {};
\vertex[fill] (r8) at (3.3,-.5) [label=below:\footnotesize{$x_4$}]{};
\vertex[fill] (r9) at (3.3,.5) [label=above:\footnotesize{$x_4$}] {};
\tikzstyle{vertex}=[circle, draw, inner sep=0pt, minimum size=0pt]
    \vertex (s) at (3.6,.5) [] {};
	\vertex (ss) at (3.6,-.5) [] {};
	\path
		(r1) edge (r2)
		(r1) edge (r3)
	   (r1) edge (r4)
	   (r1) edge (r5)
		(r2) edge (r3)
		(r2) edge (r4)
        (r2) edge (r5)
		(r3) edge (r4)
	    (r3) edge (r5)
	    (r6) edge (r4)
		(r5) edge (r7)
	    (r6) edge (r7)
		(r6) edge (r8)
	    (r6) edge (r9)
		(r8) edge (r7)
	    (r7) edge (r9)
		(r8) edge (r9)
	    (r9) edge (s)
		(r8) edge (ss);
\end{tikzpicture}
\caption{$D'_0+M''_0$ in a long block (the last two vertices have different neighbors on their right)}
\label{fig:longD4-1}
\end{figure}

\begin{proof}
If $11\leq n\leq 17$,  the only graphs of Theorem~\ref{thm:quarticOLD} which contain $H$ as a subgraph are the top two graphs of Figure~\ref{fig:n11-17} for which the two graphs on the bottom, respectively, have smaller algebraic connectives.
\begin{figure}[h!]
	\centering
\begin{tikzpicture}[scale=0.8]
	\vertex[fill] (r1) at (0,0) [] {};
	\vertex[fill] (r3) at (.5,.5) [] {};
	\vertex[fill] (r2) at (.5,-.5) [] {};
	\vertex[fill] (r5) at (1.5,.5) [] {};
	\vertex[fill] (r4) at (1.5,-.5) [] {};
\vertex[fill] (r6) at (2.3,-.5) [] {};
\vertex[fill] (r7) at (2.3,.5) [] {};
\vertex[fill] (r8) at (3.3,-.5) []{};
\vertex[fill] (r9) at (3.3,.5) [] {};
\vertex[fill] (r10) at (4.1,.5) [] {};
\vertex[fill] (r11) at (4.1,-.5) [] {};
\vertex[fill] (r12) at (5.1,.5) [] {};
\vertex[fill] (r13) at (5.1,-.5) [] {};
\vertex[fill] (r14) at (5.6,0) [] {};
\tikzstyle{vertex}=[circle, draw, inner sep=0pt, minimum size=0pt]
    \vertex () at (0,-.4) [label=below:\footnotesize{}] {};
	\path
		(r1) edge (r2)
		(r1) edge (r3)
	   (r1) edge (r4)
	   (r1) edge (r5)
		(r2) edge (r3)
		(r2) edge (r4)
        (r2) edge (r5)
		(r3) edge (r4)
	    (r3) edge (r5)
	    (r6) edge (r4)
		(r5) edge (r7)
	    (r6) edge (r7)
		(r6) edge (r8)
	    (r6) edge (r9)
		(r8) edge (r7)
	    (r7) edge (r9)
		(r8) edge (r9)
	    (r14) edge (r13)
		(r14) edge (r12)
	    (r14) edge (r11)
		(r14) edge (r10)
	   (r12) edge (r13)
	   (r11) edge (r13)
	   (r10) edge (r13)
	   (r11) edge (r12)
	   (r10) edge (r12)
	   (r11) edge (r8)
	   (r10) edge (r9) ;
\end{tikzpicture}
\quad
\begin{tikzpicture}[scale=0.8]
	\vertex[fill] (r1) at (0,0) [] {};
	\vertex[fill] (r3) at (.5,.5) [] {};
	\vertex[fill] (r2) at (.5,-.5) [] {};
	\vertex[fill] (r5) at (1.5,.5) [] {};
	\vertex[fill] (r4) at (1.5,-.5) [] {};
\vertex[fill] (r6) at (2.3,-.5) [] {};
\vertex[fill] (r7) at (2.3,.5) [] {};
\vertex[fill] (r8) at (3.3,-.5) []{};
\vertex[fill] (r9) at (3.3,.5) [] {};
\vertex[fill] (r10) at (4.2,.5) [] {};
\vertex[fill] (r11) at (4.2,-.5) [] {};
\vertex[fill] (r12) at (5.2,.5) [] {};
\vertex[fill] (r13) at (5.2,-.5) [] {};
\vertex[fill] (r14) at (6,.7) [] {};
\vertex[fill] (r15) at (6,-.7) [] {};
\vertex[fill] (r16) at (6.8,.5) [] {};
\vertex[fill] (r17) at (6.8,-.5) [] {};
	\path
		(r1) edge (r2)
		(r1) edge (r3)
	    (r1) edge (r4)
	    (r1) edge (r5)
		(r2) edge (r3)
		(r2) edge (r4)
        (r2) edge (r5)
		(r3) edge (r4)
	    (r3) edge (r5)
	    (r6) edge (r4)
		(r5) edge (r7)
	    (r6) edge (r7)
		(r6) edge (r8)
	    (r6) edge (r9)
		(r8) edge (r7)
	    (r7) edge (r9)
	    (r8) edge (r9)
		(r8) edge (r11)
	    (r9) edge (r10)
		(r10) edge (r11)
	   	(r12) edge (r10)
	     (r10) edge (r13)
	     (r11) edge (r13)
	     (r11) edge (r12)
	     (r14) edge (r12)
	     (r16) edge (r12)
	     (r13) edge (r15)
	     (r13) edge (r17)
	     (r14) edge (r15)
	     (r14) edge (r16)
	     (r14) edge (r17)
	     (r15) edge (r17)
	     (r15) edge (r16)
	     (r17) edge (r16)   ;
\end{tikzpicture}\\ \vspace{.4cm}
\begin{tikzpicture}[scale=0.8]
    \vertex[fill] (0) at (-.1,.5) []{};
	\vertex[fill] (1) at (-.1,-.5) []{};
	\vertex[fill] (2) at (.7,.7) [] {};
	\vertex[fill] (3) at (.7,-.7) [] {};
	\vertex[fill] (4) at (1.5,.5) [] {};
    \vertex[fill] (5) at (1.5,-.5)[] {};
    \vertex[fill] (6) at (2,0)[] {};
    \vertex[fill] (28) at (2.5,.5) [] {};
	\vertex[fill] (29) at (2.5,-.5) [] {};
    \vertex[fill] (00) at (3,0) [] {};
    \vertex[fill] (30) at (3.5,.7) [] {};
    \vertex[fill] (31) at (3.5,-.7)[] {};
     \vertex[fill] (32) at (4.3,.5)[] {};
    \vertex[fill] (33) at (4.3,-.5) [] {};
	\path
        (0) edge (1)
		(0) edge (2)
		(0) edge (3)
	    (0) edge (4)
		(1) edge (2)
        (1) edge (3)
		(1) edge (5)
	    (2) edge (3)
	    (2) edge (4)
	    (3) edge (5)
	    (4) edge (5)
	    (4) edge (6)
	    (5) edge (6)
	    (6) edge (28)
	    (6) edge (29)
	    (28) edge (29)
	    (28) edge (00)
	    (28) edge (30)
	    (29) edge (00)
	    (29) edge (31)
	    (00) edge (32)
	    (00) edge (33)
	    (30) edge (31)
	    (30) edge (32)
	    (30) edge (33)
	    (31) edge (32)
	    (31) edge (33)
	    (33) edge (32);
\end{tikzpicture}
\quad \quad \quad \quad
\begin{tikzpicture}[scale=0.8]
	\vertex[fill] (r) at (0,0) [] {};
	\vertex[fill] (r1) at (.5,.5) [] {};
	\vertex[fill] (r2) at (.5,-.5) [] {};
	\vertex[fill] (r3) at (1.5,.5) [] {};
	\vertex[fill] (r4) at (1.5,-.5) [] {};
    \vertex[fill] (r5) at (2,0) [] {};
   \vertex[fill] (r6) at (2.5,.5) [] {};
   \vertex[fill] (r7) at (2.5,-.5) [] {};
  \vertex[fill] (r8) at (3.5,.5) [] {};
   \vertex[fill] (r9) at (3.5,-.5) [] {};
   \vertex[fill] (r10) at (4,0) [] {};
   \vertex[fill] (r11) at (4.5,.5) [] {};
   \vertex[fill] (r12) at (4.5,-.5) [] {};
   \vertex[fill] (r13) at (5.3,.7) [] {};
   \vertex[fill] (r14) at (5.3,-.7) [] {};
    \vertex[fill] (r15) at (6.1,.5) [] {};
   \vertex[fill] (r16) at (6.1,-.5) [] {};
	\path
		(r) edge (r1)
		(r) edge (r2)
	    (r) edge (r3)
	    (r) edge (r4)
		(r1) edge (r2)
		(r1) edge (r3)
        (r1) edge (r4)
		(r2) edge (r3)
	    (r2) edge (r4)
	    (r5) edge (r4)
		(r5) edge (r3)
	    (r5) edge (r6)
		(r5) edge (r7)
	    (r6) edge (r7)
		(r6) edge (r8)
	    (r6) edge (r9)
	    (r7) edge (r8)
	    (r7) edge (r9)
	    (r9) edge (r8)
	    (r8) edge (r10)
	    (r9) edge (r10)
	    (r10) edge (r11)
	    (r10) edge (r12)
	    (r11) edge (r12)
	    (r11) edge (r13)
	    (r11) edge (r15)
        (r14) edge (r12)
	    (r16) edge (r12)
	    (r13) edge (r15)
	    (r14) edge (r13)
	    (r16) edge (r13)
	    (r14) edge (r15)
	    (r14) edge (r16)
	    (r16) edge (r15) ;
\end{tikzpicture}
\caption{Two graphs containing $H$ (top) and two graphs with smaller algebraic connectivity (bottom)}
\label{fig:n11-17}
\end{figure}
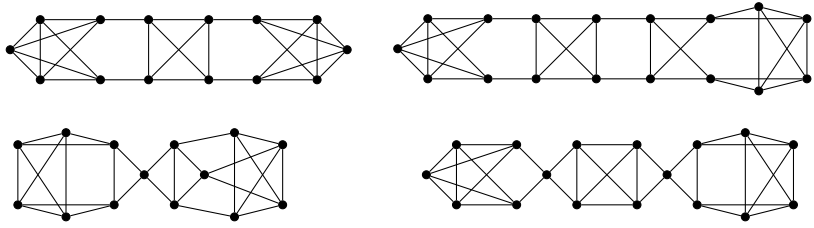
 This means that we are done for $n\le17$. So we assume that $n\ge18$, and by Lemma~\ref{lem:muUpBound},  $\mu< 0.129$.
  Now, for a contradiction assume that $\G$ contains $H$. By  using the eigen-equation \eqref{eq:eigenvector} on the first seven vertices of $H$, we obtain
\begin{equation}\label{eq:xi}
x_2=\left(\frac{-\mu}{2}+1\right)x_1,
 ~ x_3=\left(\frac{\mu^2}{2}-3\mu+1\right) x_1,
~  x_4=\left( \frac{-\mu^3+9\mu^2-19\mu}{4}+1\right)x_1.
\end{equation}\label{eq:yi}
We replace $H$ by $D_3$ to obtain $\G'$.
We define a vector $\x'$ on $V(\G')$ such that its components on the new end block $D_3$ are as given in Figure~\ref{fig:shortD3} (with $z_1,z_2,z_3$ to be specified later), and on the rest of vertices of $\G'$ agree with $\x$.
\begin{figure}
\centering
\begin{tikzpicture}[scale=0.9]
	\vertex[fill] (r1) at (0,1) [label=above:\footnotesize{$z_1$}] {};
	\vertex[fill] (r2) at (.8,1.2) [label=above:\footnotesize{$z_1$}] {};
	\vertex[fill] (r3) at (0,0) [label=below:\footnotesize{$z_1$}] {};
	\vertex[fill] (r4) at (.8,-.2) [label=below:\footnotesize{$z_1$}] {};
    \vertex[fill] (r5) at (1.6,1) [label=above:\footnotesize{$z_2$}] {};
    \vertex[fill] (r6) at (1.6,0) [label=below:\footnotesize{$z_2$}] {};
    \vertex[fill] (r7) at (2.6,1) [label=above:\footnotesize{$z_3$}] {};
    \vertex[fill] (r8) at (2.6,0) [label=below:\footnotesize{$z_3$}] {};
     \vertex[fill] (r9) at (3.1,.5) [label=above:\footnotesize{$x_4$}] {};
     \tikzstyle{vertex}=[circle, draw, inner sep=0pt, minimum size=0pt]
    \vertex (s) at (3.3,.7) [] {};
	\vertex (ss) at (3.3,.3) [] {};
	\path
		(r5) edge (r2)
		(r1) edge (r2)
	    (r1) edge (r5)
		(r1) edge (r3)
        (r1) edge (r4)
		(r2) edge (r3)
	    (r2) edge (r4)
	    (r3) edge (r6)
	    (r4) edge (r3)
	    (r4) edge (r6)
	    (r5) edge (r8)
	    (r5) edge (r7)
	    (r6) edge (r8)
	    (r6) edge (r7)
	    (r7) edge (r8)
	    (r9) edge (r7)
	    (r9) edge (r8)
	     (r9) edge (s)
	      (r9) edge (ss);
\end{tikzpicture}
\caption{The block $D_3$ in $\G'$ and the components of $\x'$}
\label{fig:shortD3}
\end{figure}
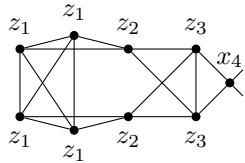

Although $\mu$ need not to be an eigenvalue of $\G'$,  we may choose  $z_1,z_2,z_3$ to satisfy the following equations which resemble the eigen-equation \eqref{eq:eigenvector} for $\mu$:
  \begin{align*}
 (1-\mu)z_1-z_2&=0, \\
(4-\mu)z_2-2z_1-2z_3&=0, \\
(3-\mu)z_3-2z_2-x_4&=0.
\end{align*}
By plugging in the value of $x_4$ from \eqref{eq:xi} and
solving the equations in $x_1$ and $\mu$,  we obtain that
\begin{equation}\label{eq:zi}
z_1=\omega x_1,\quad z_2=\omega(1-\mu)x_1,\quad z_3=\frac{\omega}{2}(\mu^2-5\mu+2)x_1,
\end{equation}
where
$$\omega=\frac{\mu^3-9\mu^2+19\mu-4}{2(\mu^3-8\mu^2+13\mu-2)}.$$
Taking into account that $\x\perp\1$, we see that
$$\delta=\x'\1^\top=4z_1+2z_2+2z_3-3x_1-2x_2-2x_3-x_4.$$
Also, as $\|\x\|=1$,
we obtain that
$$\|\x'\|^2=1+4z_1^2+2z_2^2+2z_3^2-3x_1^2-2x_2^2-2x_3^2-x_4^2.$$
 Further, we have
$$\x' L(\G')\x'^\top\!\!-\!\mu = 4(z_1-z_2)^2\!+\!4(z_2-z_3)^2\!+\!2(z_3-x_4)^2\!-\!6(x_1-x_2)^2\!-\!2(x_2-x_3)^2\!-\!4(x_3-x_4)^2.$$
 As $\|\x'\|^2-\frac{\delta^2}{n}>0$, we have  $\frac{\x'L(\G')\x'^\top}{\|\x'\|^2-\frac{\delta^2}{n}}<\mu$ if and only  if $\x'L(\G')\x'^\top-\mu\|\x'\|^2+\mu\frac{\delta^2}{n}<0$. We have
\begin{align*}
\x'L(\G')\x'^\top-\mu\|\x'\|^2+\mu\frac{\delta^2}{n}&=
12x_1x_2+4(x_2x_3+2x_3x_4)-4(z_3x_4+2z_2z_1+2z_2z_3) \\
 &\quad + \left(\mu-2\right)\left(3x_1^2+x_4^2 \right) + \left( 2\mu-8 \right)\left(x_2^2-z_2^2 \right) \\
 &\quad + \left( 2\mu-6 \right) 
 \left(x_3^2-z_3^2 \right)-\left( 4\mu-4 \right) 
z_1^2+\mu\,\frac {\delta^2}{n}.
\end{align*}
By  substituting the values of $x_2,x_3,x_4,z_1,z_2,z_3,\delta$ in terms of $x_1$ and $\mu$, we can write each term in the above as follows:
\begin{align*}
12x_1x_2&= 6(2-\mu) x_1^2,\\
4(x_2x_3+2x_3x_4)&=\left(3-\mu\right)\left( {\mu}^{2}-6\mu+2\right)^2 x_1^2, \\
-4(z_3x_4+2z_2z_1+2z_2z_3)&= \left( {\mu}^{5}-13{\mu}^{4}+59{\mu}^{3}-107{\mu}^{2}+72\mu-20 \right) \omega^2x_1^2,\\
3x_1^2+x_4^2&= \left(\mu^6-18\mu^5+119\mu^4-350\mu^3+433\mu^2-152\mu+64\right)\frac{x_1^2}{16},\\
x_2^2-z_2^2&= \mu\left(2\mu^5-30\mu^4+157\mu^3-336\mu^2+263\mu-40\right)\frac{x_1^2}{p(\mu)},\\
x_3^2-z_3^2&=(\mu\!-\!6)\mu(\mu^3\!-\!8\mu^2\!+\!12\mu\!-\!3)(3\mu^5\!-\!42\mu^4\!+\!192\mu^3\!-\!309\mu^2\!+\!134\mu\!-\!16)
\frac{x_1^2}{4p(\mu)},\\
z_1^2&= \omega^2 x_1^2,\\
\delta^2&=\mu^2 \left( \mu-5 \right)^2 \left( 
\mu-6 \right)^2
\frac{(\mu^3-8\mu^2+14\mu-5)^2}{16(\mu^3-8\mu^2+13\mu-2)^2} x_1^2,
\end{align*}
 in which
$$p(t)= 4(t-2)^2(t^2-6t+1)^2.$$
Plugging in all these terms and simplifying, it follows that $\frac{\x'L(\G')\x'^\top}{\|\x'\|^2-\frac{\delta^2}{n}}<\mu$ if and only  if \footnote{To help the reader with computations, we provide a Maple code available at \url{https://wp.kntu.ac.ir/ghorbani/ComputFiles/MapleCode.txt} that can be used to derive the rational function appeared in \eqref{eq:fracR} as well as the numeric values of the roots of $p_4(t)$. The code contains	
	similar calculations appearing in the proofs of Lemmas \ref{lem:noM2M3M4} and \ref{lem:noM5}.}
\begin{equation}\label{eq:fracR}
\frac{\left(p_1(\mu)n+p_2(\mu)\right)p_3(\mu)x_1^2}{4np(\mu)}<0,
\end{equation}
 where  
 \begin{align*}
p_1(t)&=t^6-17t^5+104t^4-275t^3+297t^2-90t+8,\\
p_2(t)&=t^6-19t^5+132t^4 -399t^3+475t^2-150t, \\
p_3(t)&=(t-6)(t-5)t^2(t^3-8t^2+14t-5).
  \end{align*}
The smallest root of $t^3-8t^2+14t-5$ is about $0.481$. As $\mu<0.129$, we have that $p_3(\mu)<0$.
The smallest root of $p_1(t)$ is about $0.171$ and so $p_1(\mu)>0$. Since $n\ge18$ we also have
$$p_1(\mu)n+p_2(\mu)\ge18p_1(\mu)+p_2(\mu)=p_4(\mu),$$
where
$$p_4(t)=19t^6-325t^5+2004t^4-5349t^3+5821t^2-1770t+144.$$
The smallest root of $p_4(t)$ is about $0.132$ which means that $p_4(\mu)>0$. Therefore,   $(p_1(\mu)n+p_2(\mu))p_3(\mu)<0$, and so \eqref{eq:fracR} is established. Hence,  by  Lemma~\ref{remark:delta},  $\mu(\G')<\mu$ which is a contradiction.
\end{proof}

\begin{lemma}\label{lem:E2}
The graph  $\G$ does not contain the graph $H$ of Figure~\ref{fig:longD4-2} as a subgraph.
\end{lemma}
\begin{figure}[H]
\centering
\begin{tikzpicture}[scale=0.9]
	\vertex[fill] (r1) at (0,0) [label=left:\footnotesize{$x_1$}] {};
	\vertex[fill] (r3) at (.5,.5) [label=above:\footnotesize{$x_1$}] {};
	\vertex[fill] (r2) at (.5,-.5) [label=below:\footnotesize{$x_1$}] {};
	\vertex[fill] (r5) at (1.5,.5) [label=above:\footnotesize{$x_2$}] {};
	\vertex[fill] (r4) at (1.5,-.5) [label=below:\footnotesize{$x_2$}] {};
\vertex[fill] (r6) at (2.3,-.5) [label=below:\footnotesize{$x_3$}] {};
\vertex[fill] (r7) at (2.3,.5) [label=above:\footnotesize{$x_3$}] {};
\vertex[fill] (r8) at (3.3,-.5) [label=below:\footnotesize{$x_4$}] {};
\vertex[fill] (r9) at (3.3,.5) [label=above:\footnotesize{$x_4$}] {};
\tikzstyle{vertex}=[circle, draw, inner sep=0pt, minimum size=0pt]
    \vertex (s) at (3.6,.5) [] {};
	\vertex (ss) at (3.6,-.5) [] {};
    \vertex (sss) at (3.6,.3) [] {};
	\vertex (ssss) at (3.6,-.3) [] {};
	\path
		(r1) edge (r2)
		(r1) edge (r3)
	   (r1) edge (r4)
	   (r1) edge (r5)
		(r2) edge (r3)
		(r2) edge (r4)
        (r2) edge (r5)
		(r3) edge (r4)
	    (r3) edge (r5)
	    (r6) edge (r4)
		(r5) edge (r7)
	    (r6) edge (r7)
		(r6) edge (r8)
	    (r6) edge (r9)
		(r8) edge (r7)
	    (r7) edge (r9)
		(r8) edge (ssss)
	    (r9) edge (sss)
	    (r9) edge (s)
		(r8) edge (ss);
\end{tikzpicture}
\caption{$D'_0+M''_1$ or $D'_0+\tilde M'_2$  in a long block}\label{fig:longD4-2}
\end{figure}
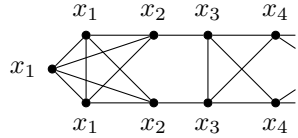
\begin{proof}
For a contradiction suppose that $\G$ contains $H$.
So $n\geq 13$ (otherwise $\G$ cannot contain $H$).
We may assume that $x_1>0$ (since otherwise we consider $-\x$).
By using the eigen-equation \eqref{eq:eigenvector}, we have
$$x_3=(\mu^2-6\mu+2)\frac{x_1}2.$$
As $\mu<0.268$ (by Lemma~\ref{lem:muUpBound}),  we have $\mu^2-6\mu+2>0$ which implies that $x_3>0$.

Now, we replace $H$ by $H'$ to obtain $\G'$
and define a vector $\x'$ on $V(\G')$ such that its components on $H'$ are as  in Figure~\ref{fig:H'inG'}, and on the rest of vertices agree with $\x$.
\begin{figure}[H]
\centering
\begin{tikzpicture}[scale=0.9]
	\vertex[fill] (r1) at (0,.8) [label=above:\footnotesize{$x_1$}] {};
	\vertex[fill] (r2) at (.8,1) [label=above:\footnotesize{$x_1$}] {};
	\vertex[fill] (r3) at (0,0) [label=below:\footnotesize{$x_1$}] {};
	\vertex[fill] (r4) at (.8,-.2) [label=below:\footnotesize{$x_1$}] {};
    \vertex[fill] (r5) at (1.6,.8) [label=above:\footnotesize{$x_2$}] {};
    \vertex[fill] (r6) at (1.6,0) [label=below:\footnotesize{$x_2$}] {};
    \vertex[fill] (r7) at (2.3,.4) [label=above:\footnotesize{$x_3$}] {};
    \vertex[fill] (r8) at (3,.8) [label=above:\footnotesize{$x_4$}] {};
    \vertex[fill] (r9) at (3,0) [label=below:\footnotesize{$x_4$}] {};
\tikzstyle{vertex}=[circle, draw, inner sep=0pt, minimum size=0pt]
    \vertex (s) at (3.3,.8) [] {};
	\vertex (ss) at (3.3,0) [] {};
    \vertex (sss) at (3.3,.6) [] {};
	\vertex (ssss) at (3.3,.2) [] {};
	\path
		(r5) edge (r2)
		(r1) edge (r2)
	    (r1) edge (r5)
		(r1) edge (r3)
        (r1) edge (r4)
		(r2) edge (r3)
	    (r2) edge (r4)
	    (r3) edge (r6)
	    (r4) edge (r3)
	    (r4) edge (r6)
	    (r5) edge (r6)
	    (r7) edge (r6)
	    (r5) edge (r7)
	    (r9) edge (r8)
	    (r7) edge (r8)
	    (r7) edge (r9)
	    (r8) edge (s)
		(r8) edge (sss)
	    (r9) edge (ss)
		(r9) edge (ssss);
\end{tikzpicture}
\caption{The subgraph $H'$ in $\G'$ and the components of $\x'$}
\label{fig:H'inG'}
\end{figure}
By the fact that $\x$ is a unit Fiedler vector of $\G$,  we  see that
$$\delta=\x'\1^\top=x_1-x_3, \quad \|\x'\|^2=1+x_1^2-x_3^2,\quad \x'L(\G')\x'^\top=\mu-2(x_1-x_2)^2-2(x_3-x_4)^2.$$
The situation here is similar to the proof of Lemma~\ref{lem:E3}. Therefore, we can deduce that $\mu(\G')<\mu$ which is a contradiction.
\end{proof}

All in all, we have proved the following theorem:

 \begin{theorem} \label{thm:endblocks} The end blocks of $\G$ belong to $D_0,\ldots,D_4$.
\end{theorem}

We remark that Theorem \ref{thm:endblocks},  in particular, implies that $\G$ cannot be a long complete block.

\subsection{Middle blocks}

Our goal is to show that beside $M_0$, $\G$ contains no other types of middle blocks.
The following  lemma provides a key tool for our arguments.

 \begin{lemma}\label{rem:main1}
Let $\rho$ be a unit Fiedler vector of a quartic graph $G$ of order $n$, and $H$ be a subgraph of $G$ such that $G-H$ has two connected components $G_1$ and $G_2$. Let $H'$ be a graph with the same number of vertices and edges as $H$.
  In $G$ we replace $H$ by $H'$ to obtain a new graph $G'$ such that $G'$ is also a quartic graph.
   Suppose that there is a vector $\rho'$ on $V(G')$ such that %$\rho'\perp\bf1$ and
$$\ell:=\sum_{ij\in E(G)\setminus E(H)}(\rho_i-\rho_j)^2=\sum_{ij\in E(G')\setminus E(H')}(\rho'_i-\rho'_j)^2.$$
Set  $$h:=\sum_{ij\in E(H)}(\rho_i-\rho_j)^2,~~h':=\sum_{ij\in E(H')}(\rho'_i-\rho'_j)^2,~~\hbox{and}~~\epsilon:=\|\rho'\|^2-\frac{\delta^2}{n}-1,$$
where  $\delta=\rho'\1^\top$. If $h'-h-\epsilon\mu(G)<0,$ then $\mu(G')<\mu(G).$
\end{lemma}
\begin{proof}
We have  $\mu(G)=\rho L(G)\rho^\top=\ell+h$, and by Lemma~\ref{remark:delta},
$$ \mu(G')\le\frac{\rho' L(G')\rho'^\top}{\|\rho'\|^2-\frac{\delta^2}{n}}=\frac{\ell+h'}{1+\epsilon}.$$
It follows that if $h'-h-\epsilon\mu(G)<0,$ then $\mu(G')<\mu(G).$
\end{proof}

\begin{lemma}\label{lem:noM2M3M4}
The graph $\G$ does not contain the  graphs of Figure~\ref{fig:H1H2H3} as induced subgraphs with the given conditions on the components of a Fielder vector:
\begin{itemize}
	\item[\rm(i)]   $H_1$ such that $x_{r+3}\ge0$,
	\item[\rm(ii)]   $H_2$ such that $x_{r+4}\ge0$,
	\item[\rm(iii)]   $H_3$ such that $x_{r+3}\ge0$.	
\end{itemize}
\end{lemma}
\begin{figure}[H]
\captionsetup[subfigure]{labelformat=empty}
\centering
\subfloat[$H_1$]{\begin{tikzpicture}[scale=0.9]
	\vertex[fill] (r1) at (.3,.5) [label=above:\footnotesize{$x_r$}] {};
	\vertex[fill] (r2) at (.3,-.5) [] {};
    \vertex[fill] (r3) at (.9,0) [] {};
    \vertex[fill] (r4) at (1.5,.5) [label=above:\footnotesize{$x_{r+2}$}] {};
    \vertex[fill] (r5) at (1.5,-.5) [] {};
    \vertex[fill] (r6) at (2.5,.5) [label=above:\footnotesize{$x_{r+3}$}] {};
    \vertex[fill] (r7) at (2.5,-.5) [] {};
    \tikzstyle{vertex}=[circle, draw, inner sep=0pt, minimum size=0pt]
     \vertex[] () at (.92,-.15) [label=below:\footnotesize{$x_{r+1}$}] {};
	\path
		(r1) edge (r2)
		(r1) edge (r3)
		(r2) edge (r3)
	    (r3) edge (r4)
	    (r3) edge (r5)
	    (r4) edge (r5)
	    (r4) edge (r6)
	    (r4) edge (r7)
	    (r5) edge (r7)
	    (r5) edge (r6)  ;
\end{tikzpicture}}
\qquad
\subfloat[$H_2$]{\begin{tikzpicture}[scale=0.9]
	\vertex[fill] (r1) at (.3,.5) [label=above:\footnotesize{$x_r$}] {};
	\vertex[fill] (r2) at (.3,-.5) [] {};
    \vertex[fill] (r3) at (.9,0) [] {};
    \vertex[fill] (r4) at (1.5,.5) [label=above:\footnotesize{$x_{r+2}$}] {};
    \vertex[fill] (r5) at (1.5,-.5) [] {};
    \vertex[fill] (r6) at (2,0) [label=right:\footnotesize{$x_{r+3}$}] {};
    \vertex[fill] (r7) at (2.5,.5) [label=above:\footnotesize{$x_{r+4}$}] {};
    \vertex[fill] (r8) at (2.5,-.5) [] {};
      \tikzstyle{vertex}=[circle, draw, inner sep=0pt, minimum size=0pt]
    \vertex[] () at (.92,-.15) [label=below:\footnotesize{$x_{r+1}$}] {};
	\path
		(r1) edge (r2)
		(r1) edge (r3)
		(r2) edge (r3)
	    (r3) edge (r4)
	    (r3) edge (r5)
	    (r4) edge (r5)
	    (r4) edge (r6)
	    (r4) edge (r7)
	    (r5) edge (r6)
	    (r5) edge (r8)
	    (r6) edge (r7)
	    (r6) edge (r8);
\end{tikzpicture}}
\qquad
\subfloat[$H_3$]{\begin{tikzpicture}[scale=0.9]
	\vertex[fill] (r1) at (.3,.5) [label=above:\footnotesize{$x_r$}] {};
	\vertex[fill] (r2) at (.3,-.5) [] {};
    \vertex[fill] (r3) at (.9,0) [] {};
    \vertex[fill] (r4) at (1.5,.5) [label=above:\footnotesize{$x_{r+2}$}] {};
    \vertex[fill] (r5) at (1.5,-.5) [] {};
    \vertex[fill] (r6) at (2.5,.5) [label=above:\footnotesize{$x_{r+3}$}] {};
    \vertex[fill] (r7) at (2.5,-.5) [] {};
     \vertex[fill] (r8) at (3.5,.5) [label=above:\footnotesize{$x_{r+4}$}] {};
     \vertex[fill] (r9) at (3.5,-.5) [] {};
     \tikzstyle{vertex}=[circle, draw, inner sep=0pt, minimum size=0pt]
     \vertex[] () at (.92,-.15) [label=below:\footnotesize{$x_{r+1}$}] {};
	\path
		(r1) edge (r2)
		(r1) edge (r3)
		(r2) edge (r3)
	    (r3) edge (r4)
	    (r3) edge (r5)
	    (r4) edge (r5)
	    (r4) edge (r6)
	    (r4) edge (r7)
	    (r5) edge (r7)
	    (r5) edge (r6)
	    (r7) edge (r6)
        (r8) edge (r6)
	    (r7) edge (r9)
	     (r8) edge (r9)	 ;
\end{tikzpicture}}
\caption{Some forbidden subgraphs for $\G$}\label{fig:H1H2H3}
\end{figure}
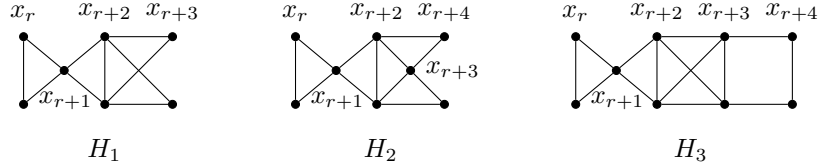
\begin{proof} (i)	By contradiction, let $H_1$ be a subgraph of $\G$. % Recall that $x_r>x_{r+1}>x_{r+2}>x_{r+3}\ge0$. 	
	By  the eigen-equation \eqref{eq:eigenvector} considered on $H_1$, it can be seen that $x_{r+1}=f(x_r,x_{r+3})$ and $x_{r+2}=g(x_r,x_{r+3})$,
	where
	\begin{align*}
	f(x,y)&=\frac {2(3x-\mu x+2y)}{\mu^2-7\mu+10},\\
	g(x,y)&=\frac {2(x-\mu y+4y)}{\mu^2-7\mu+10}.	
	\end{align*}
	We	replace $H_1$ by $H'_1$ to obtain $\G'$ and
	define a vector $\x'$ on $V(\G')$ such that its components on $H'_1$ are as given in Figure~\ref{fig:H'1inG'} (with $z_{r+1},z_{r+2}$ to be specified later) and on the rest of vertices of $\G'$ agree with $\x$.	
\begin{figure}[H]
\centering
\begin{tikzpicture}[scale=0.9]
			\vertex[fill] (r1) at (1.5,.5)[label=above:\footnotesize{$x_r$}] {};
			\vertex[fill] (r2) at (1.5,-.5) [] {};
			\vertex[fill] (r3) at (2.5,.5) [label=above:\footnotesize{$z_{r+1}$}] {};
			\vertex[fill] (r4) at (2.5,-.5)[] {};
			\vertex[fill] (r5) at (3.1,0)[] {};
			\vertex[fill] (r6) at (3.7,.5)[label=above:\footnotesize{$x_{r+3}$}] {};
			\vertex[fill] (r7) at (3.7,-.5) [] {};
	     \tikzstyle{vertex}=[circle, draw, inner sep=0pt, minimum size=0pt]
         \vertex[] () at (3.12,-.15) [label=below:\footnotesize{$z_{r+2}$}] {};
			\path
			(r1) edge (r3)
			(r1) edge (r4)
			(r2) edge (r3)
			(r2) edge (r4)
			(r3) edge (r4)
			(r3) edge (r5)
			(r4) edge (r5)
			(r5) edge (r6)
			(r5) edge (r7)
			(r6) edge (r7);
			\end{tikzpicture}
\caption{The subgraph $H'_1$ in $\G'$ and the components of $\x'$}
\label{fig:H'1inG'}
\end{figure}
We set
$$z_{r+1}=g(x_{r+3},x_r), \quad z_{r+2}=f(x_{r+3},x_r).$$
We  have $\delta=\x'\1^\top=2z_{r+1}+z_{r+2}-x_{r+1}-2x_{r+2}$. By substituting the values of $x_{r+1},x_{r+2},z_{r+1},z_{r+2}$ in terms of $x_r,x_{r+3}, \mu$ we obtain that
$$\delta=\frac {2(x_r-x_{r+3})}{2-\mu}.$$
	Now, in the  notation of  Lemma~\ref{rem:main1}, we have:
	\begin{align*}
	h&=2(x_r-x_{r+1})^2+2(x_{r+1}-x_{r+2})^2+4(x_{r+2}-x_{r+3})^2,\\
	h'&=4(x_r-z_{r+1})^2+2(z_{r+1}-z_{r+2})^2+2(z_{r+2}-x_{r+3})^2.
	\end{align*}
	Moreover, since $\|\x\|=1$,
$$\epsilon=\|\x'\|^2-1-\frac{\delta^2}n=2z_{r+1}^2+z_{r+2}^2-x_{r+1}^2-2x_{r+2}^2-\frac{\delta^2}n.$$
	It follows that
\begin{align*}
h'-h-\epsilon\mu&= 2\left( x_r+2x_{r+1}-4z_{r+1} \right) x_{{r}}+ 4(z_{r+1}^2-x_{r+1}^2)+ 4(x_{r+1}x_{r+2}-z_{r+1}z_{r+2}) \\
&\quad - 6( x_{r+2}^2-z_{r+1}^2)+ 2x_{r+3} \left( 4x_{r+2}-x_{r+3}-2z_{r+2} \right) -\epsilon\mu\\
&=\frac{2\mu+4}{\mu-2}x_r^2+\frac{16(1-\mu)(x_r^2-x_{r+3}^2)}{(\mu-2)^2(\mu-5)}
+\frac{16(x_r^2-x_{r+3}^2)}{(\mu-2)^2(\mu-5)} \\
&\quad +\frac{24(\mu-3)(x_r^2-x_{r+3}^2)}{(\mu-2)^2(\mu-5)}-\frac{2\mu+4}{\mu-2}x_{r+3}^2 
-\frac{4\mu(x_r^2-x_{r+3}^2)}{(\mu-2)^2}+\frac{\mu\delta^2}{n}\\ 
&=\frac { 2\mu(x_r-x_{r+3})}{(\mu-2)^2n} \big( (\mu-2) n(x_r+x_{r+3})+2(x_r-x_{r+3})\big),
\end{align*}
which is negative because $x_r>x_{r+3}\geq 0$, $n\geq 11$, and $\mu<0.355$.      
Therefore, from  Lemma~\ref{rem:main1} it follows  that $\mu(\G')<\mu$, a contradiction.

 (ii)	For a contradiction, assume that $\G$ contains $H_2$.
	By  the eigen-equation \eqref{eq:eigenvector} considered on the vertices of $H_2$, it can be seen that $x_{r+1}=f(x_r,x_{r+4})$, $x_{r+2}=g(x_r,x_{r+4})$, and $x_{r+3}=l(x_r,x_{r+4})$, where
\begin{align*}
f(x,y)&=\frac {2(-\mu^2 x+7\mu x+\mu y-10x-6y)}{ \mu^3-11\mu^2+36\mu-32},\\
g(x,y)&=\frac {-\mu y+2x+6y}{{\mu}^{2}-7\mu+8}, \\
l(x,y)&=\frac {2(-\mu^2 y+8\mu y-2x-14y)}{ \mu^3-11\mu^2+36\mu-32}.
\end{align*}	
We	replace $H_2$ by $H'_2$ to obtain $\G'$ and
	define a vector $\x'$ on $V(\G')$ such that its components on $H'_2$ are as given in Figure~\ref{fig:H'2inG'} and on the rest of vertices agree with $\x$.	
\begin{figure}[H]
\centering
\begin{tikzpicture}[scale=0.9]
			\vertex[fill] (r1) at (1.5,.5)[label=above:\footnotesize{$x_r$}] {};
			\vertex[fill] (r2) at (1.5,-.5) [] {};
			\vertex[fill] (r3) at (2,0) [label=left:\footnotesize{$z_{r+1}$}] {};
			\vertex[fill] (r4) at (2.5,.5) [label=above:\footnotesize{$z_{r+2}$}] {};
			\vertex[fill] (r5) at (2.5,-.5)[] {};
			\vertex[fill] (r6) at (3.1,0) [] {};
			\vertex[fill] (r7) at (3.7,.5)[label=above:\footnotesize{$x_{r+4}$}] {};
			\vertex[fill] (r8) at (3.7,-.5) [] {};
	\tikzstyle{vertex}=[circle, draw, inner sep=0pt, minimum size=0pt]
     \vertex[] () at (3.12,-.15) [label=below:\footnotesize{$z_{r+3}$}] {};
			\path
			(r1) edge (r3)
			(r1) edge (r4)
			(r2) edge (r3)
			(r2) edge (r5)
			(r3) edge (r4)
			(r3) edge (r5)
			(r4) edge (r5)
			(r4) edge (r6)
			(r5) edge (r6)
			(r6) edge (r7)
			(r6) edge (r8)
			(r8) edge (r7);
			\end{tikzpicture}
\caption{The subgraph $H'_2$ in $\G'$ and the components of $\x'$}
\label{fig:H'2inG'}
\end{figure}
We set
$$z_{r+1}=l(x_{r+4},x_r), \quad z_{r+2}=g(x_{r+4},x_r), \quad z_{r+3}=f(x_{r+4},x_r).$$
We have $\delta=\x'\1^\top=z_{r+1}+2z_{r+2}+z_{r+3}-x_{r+1}-2x_{r+2}-x_{r+3}$. By substituting the values of $x_{r+1},x_{r+2},x_{r+3},z_{r+1},z_{r+2},z_{r+3}$ in terms of $x_r,x_{r+4},\mu$, we obtain that
$$\delta=\frac {2(6-\mu)(x_r-x_{r+4})}{\mu^2-7\mu+8}.$$
	Now, in the notations of  Lemma~\ref{rem:main1}, we have:
	\begin{align*}
	h&=2\sum_{i=r}^{r+3}(x_i-x_{i+1})^2+2(x_{r+2}-x_{r+4})^2,  \\
	h'&=2(x_{r}-z_{r+1})^2+2(x_{r}-z_{r+2})^2+2(z_{r+1}-z_{r+2})^2+2(z_{r+2}-z_{r+3})^2+2(z_{r+3}-x_{r+4})^2.
	\end{align*}
	Moreover,
$$\epsilon=\|\x'\|^2-1-\frac{\delta^2}n=z_{r+1}^2+2z_{r+2}^2+z_{r+3}^2-x_{r+1}^2-2x_{r+2}^2-x_{r+3}^2-\frac{\delta^2}n.$$
It follows that
\begin{align*}
h'-h-\epsilon\mu &=2x_{{r}} \left( x_{r}+2x_{r+1}-2z_{r+1}-2z_{r+2}\right)-6(x_{r+2}^2-z_{r+2}^2)\\
&\quad -4(x_{r+1}^{2}+x_{r+3}^{2}-z_{r+1}^{2}-z_{r+3}^{2})+2x_{r+4} \left(2x_{{r+3}}-x_{r+4}-2z_{r+3} \right) \\ 
&\quad +4 x_{r+2}\left( x_{{r+1}}+x_{{r+3}}+x_{{r+4}} \right)-4z_{{r+2}}\left( z_{{r+1}}+z_{{r+3}} \right)-\epsilon\mu \\
&=\frac{2}{\omega}(\mu^2-5\mu-8)x_r^2+\frac{6}{\omega^2}(\mu^2-12\mu+32)(x_r^2-x_{r+4}^2)
-\frac{32}{\omega^2}(\mu-4)(x_r^2-x_{r+4}^2)\\
&\quad -\frac{2}{\omega}\left(\left(\mu^{2}-7\mu+4 \right) x_{r+4}+4x_r\right)x_{r+4} -\frac{4}{\omega^2}
\left(\left(\mu-6\right) x_{{r+4}}-2x_{{r}}\right)^{2}\left(\mu-3\right)\\
&\quad-\frac{8}{\omega^2} \left(\mu x_{{r}}-6 x_{{r}}-2x_{{r+4}}\right)\left(\left( x_{{r}}+x_{{r+4}}\right) \mu-5x_{r}-3x_{r+4}\right) \\
&\quad-\frac{2}{\omega^2}(\mu^2-16\mu+48)\mu (x_r^2-x_{r+4}^2)+\frac{\mu\delta^2}{n}\\
&=\Phi\frac{2\mu(\mu-6)(x_r-x_{r+4})}{\omega^2n},
\end{align*}
where $\omega=\mu^2-7\mu+8$ and $\Phi=(\mu^2-7\mu+8)n(x_r+x_{r+4})+(2\mu-12)(x_r-x_{r+4}).$
It is easy to check that $\Phi>0$ since $x_r>x_{r+4}\geq 0$, $n\geq 11$, and  $\mu<0.355$.
It follows that $h'-h-\epsilon\mu<0$ and so by Lemma~\ref{rem:main1},  $\mu(\G')<\mu$, a contradiction.

 (iii)	 For a contradiction, assume that $\G$ contains $H_3$. It should have at least $21$ vertices to contain $H_3$.
So $n\ge21$, and by Lemma~\ref{lem:muUpBound}, $\mu<0.091$.
	By  the eigen-equation \eqref{eq:eigenvector} considered on the vertices of $H_3$, it can be seen that $x_{r+1}=f(x_r,x_{r+4})$, $x_{r+2}=g(x_r,x_{r+4})$, and $x_{r+3}=l(x_r,x_{r+4})$,
	where
	\begin{align*}
	f(x,y)&=\frac{2(-\mu^2 x+6\mu x-5x-2y)}{\mu^3-10\mu^2+27\mu-14},\\
	g(x,y)&=\frac{2(\mu x+\mu y-3x-4y)}{\mu^3-10\mu^2+27\mu-14},  \\
	l(x,y)&=\frac{-\mu^2y+7\mu y-4x-10y}{\mu^3-10\mu^2+27\mu-14}.
	\end{align*}
We have that $x_r>x_{r+4}$ (by Remark~\ref{rem:sign}). We further claim that
\begin{equation}\label{eq:xr+xr+4}
3(x_r+x_{r+4})>x_r-x_{r+4}>0.
\end{equation}
If $x_{r+4}\ge0$, this trivially holds.
If $x_{r+4}<0$, it holds by the following argument.	Note that $l(x_r,x_{r+4})=x_{r+3}\ge0$ which means $-\mu^2x_{r+4}+7\mu x_{r+4}-4x_r-10x_{r+4}\le0$. Therefore,
	$4x_r\ge(-\mu^2+7\mu-10)x_{r+4}>-8x_{r+4}$, that is $x_r>-2x_{r+4}$, from which \eqref{eq:xr+xr+4} follows.

We	now replace $H_3$ by $H'_3$ to obtain $\G'$ and
	define a vector $\x'$ on $V(\G')$ such that its components on $H'_3$ are as given in Figure~\ref{fig:H'3inG'} and on the rest of vertices agree with $\x$.	
\begin{figure}[H]
\centering
\begin{tikzpicture}[scale=0.9]
			\vertex[fill] (r) at (.5,.5)[label=above:\footnotesize{$x_r$}] {};
			\vertex[fill] (r0) at (.5,-.5) [] {};
			\vertex[fill] (r1) at (1.5,.5)[label=above:\footnotesize{$z_{r+1}$}] {};
			\vertex[fill] (r2) at (1.5,-.5) [] {};
			\vertex[fill] (r3) at (2.5,.5) [label=above:\footnotesize{$z_{r+2}$}] {};
			\vertex[fill] (r4) at (2.5,-.5)[] {};
			\vertex[fill] (r5) at (3.1,0) [] {};
			\vertex[fill] (r6) at (3.7,.5)[label=above:\footnotesize{$x_{r+4}$}] {};
			\vertex[fill] (r7) at (3.7,-.5) [] {};
	\tikzstyle{vertex}=[circle, draw, inner sep=0pt, minimum size=0pt]
     \vertex[] () at (3.1,-.15) [label=below:\footnotesize{$z_{r+3}$}] {};
			\path
			(r) edge (r0)
			(r) edge (r1)
			(r0) edge (r2)
			(r1) edge (r2)
			(r1) edge (r4)	
			(r1) edge (r3)
			(r2) edge (r3)
			(r2) edge (r4)
			(r3) edge (r4)
			(r3) edge (r5)
			(r4) edge (r5)
			(r5) edge (r6)
			(r5) edge (r7)
			(r6) edge (r7);
			\end{tikzpicture}
\caption{The subgraph $H'_3$ in $\G'$ and the components of $\x'$}
\label{fig:H'3inG'}
\end{figure}
We set
$$z_{r+1}=l(x_{r+4},x_r), \quad z_{r+2}=g(x_{r+4},x_r), \quad z_{r+3}=f(x_{r+4},x_r).$$
We  have $\delta=\x'\1^\top=2z_{r+1}+2z_{r+2}+z_{r+3}-x_{r+1}-2x_{r+2}-2x_{r+3}$. By substituting the values of $x_{r+1},x_{r+2},x_{r+3},z_{r+1},z_{r+2},z_{r+3}$ in terms of $x_r,x_{r+4},\mu$ we obtain that
$$\delta=\frac {2(\mu-5)(x_{{k}}-\,x_{{k+4}})}{\mu^3-10\mu^2+27\mu-14}.$$
	Now, in the notations of  Lemma~\ref{rem:main1}, we have:
	\begin{align*}
	h&=2\sum_{i=r}^{r+3}(x_i-x_{i+1})^2+2(x_{r+2}-x_{r+3})^2,  \\
	h'&=2(x_{r}-z_{r+1})^2+4(z_{r+1}-z_{r+2})^2+2(z_{r+2}-z_{r+3})^2+2(z_{r+3}-x_{r+4})^2.
	\end{align*}
	Moreover,
$$\epsilon=\|\x'\|^2-1-\frac{\delta^2}n=2z_{r+1}^2+2z_{r+2}^2+z_{r+3}^2-x_{r+1}^2-2x_{r+2}^2-2x_{r+3}^2-\frac{\delta^2}n.$$
	It follows that
\begin{align*}
h'-h-\epsilon\mu &=4\left( x_{{r+1}}-z_{{r+1}} \right) x_{{r}}
-6(x_{r+2}^{2}+x_{r+3}^{2}-z_{r+1}^{2}-z_{r+2}^{2})\\
&\quad -4z_{{r+3}}\left( x_{{r+4}}+z_{{r+2}}-z_{{r+3}}\right)+8(x_{{r+2}}x_{{r+3}}-z_{{r+1}}z_{{r+2}})\\
&\quad+4(x_{r+1}x_{r+2}+x_{r+3}x_{r+4}-x_{r+1}^2)-\epsilon\mu\\
&=\Phi\frac{-2\mu(\mu-5)(x_r-x_{r+4})}{(\mu^3-10\mu^2+27\mu-14)^2n},
\end{align*}
	in which
	$$\Phi=(\mu^3-10\mu^2+27\mu-14)n(x_r+x_{r+4})+(10-2\mu)(x_r-x_{r+4}).$$
The only real zero of $t^3-10t^2+27t-4$ is about $0.157$.
Since $\mu<0.091$, we have $\mu^3-10\mu^2+27\mu-4<0$.
It follows that
\begin{equation}\label{eq:Phi}
\Phi<-10n(x_r+x_{r+4})+10(x_r-x_{r+4}).
\end{equation}
In view of \eqref{eq:xr+xr+4}, the right hand side of \eqref{eq:Phi} is negative, and thus $h'-h-\epsilon\mu<0$. The result now follows from  Lemma~\ref{rem:main1}.
\end{proof}

\begin{lemma}\label{lem:M5orM'1}
Any middle block of $\G$ is either $M_0$ or $M_3$.
\end{lemma}
\begin{proof} Let $B$ be an arbitrary middle block of $\G$.
According to Theorem~\ref{thm:quarticOLD}, $B$ is either one of the short blocks $M_0, M_1, M_2, \tilde M_2, M_3$,
 or it is a long block starting with either $M'_0, M'_1$, or $M'_2$.

First assume that $B$ is either a long block starting with $M'_1$, or it is one of the short blocks $M_1$ or $\tilde M_2$.
Then $\G$ contains $H_1$ (of Lemma~\ref{lem:noM2M3M4}\,(i)) as an induced subgraph.
If all the components of $\x$ on the first five vertices of $B$ are non-negative, then
 the required condition of Lemma~\ref{lem:noM2M3M4}\,(i) is satisfied, and so we are done.
  Otherwise, $\x$ is negative on $B$ except possibly for its first three vertices.
Recall that if $B$ is a long block, then its last brick $B_s$ is either $\tilde M'_0, \tilde M'_1$, or $\tilde M'_2$.
 Now consider the mirror image $\tilde\G$ and $-\x$ as its Fiedler vector.
It turns out that in $\tilde\G$, $-\x$ is non-negative on  $\tilde B$ except possibly for its last three vertices.
 If $B=M_1$ or $B_s=\tilde M'_1$, then $\tilde\G$ contains $H_1$;
  if $B=\tilde M_2$ or $B_s=\tilde M'_2$, then $\tilde\G$ contains $H_2$; and if $B_s=\tilde M'_0$, then $\tilde\G$ contains $H_3$.
  Furthermore, the required conditions on the sign of the components of the Fiedler vector in Lemma~\ref{lem:noM2M3M4} holds in all these three cases. This leads again to a contradiction by Lemma~\ref{lem:noM2M3M4}.

If $B$ is long block starting with a $M'_2$ or $B=M_2$, we obtain a contradiction similarly.

So far we have proved that if $B$ is a long middle block, then it must start and end with $M'_0$ and $\tilde M'_0$, respectively.
Such a block must contain $H_3$. If the components of $\x$ satisfy the condition (iii) of  Lemma~\ref{lem:noM2M3M4}, then we are done. Otherwise, $x_{r+3}<0$ and so $\tilde B$ in $\tilde\G$ fulfills the condition (iii), and again we reach a contradiction.

It follows that $B$ must be one of the short blocks $M_0$ or $M_3$, as desired.
\end{proof}

It remains to show that no middle block in $\G$ can be $M_3$. This will be essentially done by the next lemma.
By  $D_iM_3$ we denote the graph obtained by identifying the degree $2$ vertices of $D_i$ and $M_3$.

\begin{lemma}\label{lem:noM5}
	The graph $\G$ does not contain the following subgraphs:
	\begin{itemize}
	\item[\rm(i)]   $H_4$ (of Figure~\ref{fig:H4}) such that $x_{r+4}\geq 0$ and for the two left-most vertices, all their neighbors on their left lie in a single cell of the equitable partition of $\G$,
		\item[\rm(ii)]  $D_0M_3$,
		\item[\rm(iii)]   $D_2M_3$.
	\end{itemize}
\end{lemma}
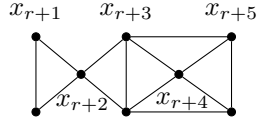
\begin{figure}[H]
\centering
\begin{tikzpicture}[scale=1]
	\vertex[fill] (r3) at (1.3,.5) [label=above:\footnotesize{$x_{r+1}$}] {};
	\vertex[fill] (r4) at (1.3,-.5) [] {};
	\vertex[fill] (r5) at (1.9,0) [] {};
	\vertex[fill] (r6) at (2.5,.5) [label=above:\footnotesize{$x_{r+3}$}] {};
	\vertex[fill] (r7) at (2.5,-.5) [] {};
	\vertex[fill] (r8) at (3.2,0) [] {};
	\vertex[fill] (r9) at (3.9,.5) [label=above:\footnotesize{$x_{r+5}$}] {};
	\vertex[fill] (r10) at (3.9,-.5) [] {};
  \tikzstyle{vertex}=[circle, draw, inner sep=0pt, minimum size=0pt]
    \vertex[] () at (1.92,-.12) [label=below:\footnotesize{$x_{r+2}$}] {};
     \vertex[] () at (3.25,-.09) [label=below:\footnotesize{$x_{r+4}$}] {};
	\path
	(r4) edge (r3)
	(r3) edge (r5)
	(r4) edge (r5)
	(r5) edge (r7)
	(r5) edge (r6)
	(r7) edge (r6)
	(r8) edge (r6)
	(r6) edge (r9)
	(r8) edge (r7)
	(r10) edge (r7)
	(r8) edge (r9)
	(r8) edge (r10)
	(r9) edge (r10) ;
	\end{tikzpicture}
\caption{The subgraph $H_4$ and the components of $\x$}
\label{fig:H4}
\end{figure}
\begin{proof}
 (i) For a contradiction, assume that $\G$ contains $H_4$.
 By the assumption, we can assume that for the two left-most vertices, all their neighbors on their left have the weight $x_r$.
	Applying  the eigen-equation \eqref{eq:eigenvector} to the vertices of $H_4$ in $\G$, we see that
$$x_{r+1}=f(x_r,x_{r+5}),~ x_{r+2}=g(x_r,x_{r+5}),~  x_{r+3}=l(x_r,x_{r+5}),~ x_{r+4}=p(x_r,x_{r+5}),$$ where
	\begin{align*}
	f(x,y)&={\frac {-2({\mu}^{3}x-11\,{\mu}^{2}x+36\,\mu\,x+\mu\,y-32\,x-6\,y)}{{
				\mu}^{4}-14\,{\mu}^{3}+67\,{\mu}^{2}-126\,\mu+76}},\\
	g(x,y)&={\frac {2(2\,{\mu}^{2}x+{\mu}^{2}y-14\,\mu\,x-9\,\mu\,y+20\,x+18\,y)}{
			{\mu}^{4}-14\,{\mu}^{3}+67\,{\mu}^{2}-126\,\mu+76}}, \\
	l(x,y)&=-{\frac {{\mu}^{3}y-13\,{\mu}^{2}y+4\,\mu x+52\,\mu\,y-16\,x-60\,y}{{\mu}^{4}-
			14\,{\mu}^{3}+67\,{\mu}^{2}-126\,\mu+76}},   \\
	p(x,y)&={\frac {-2({\mu}^{3}y-11\,{\mu}^{2}y+36\,\mu\,y-34\,y-4\,x)}{{\mu}^{4}
			-14\,{\mu}^{3}+67\,{\mu}^{2}-126\,\mu+76}}.
	\end{align*}
We have that $x_r>x_{r+5}$ (by Remark~\ref{rem:sign}). We further claim that
\begin{equation}\label{eq:xr+xr+5}
2(x_r+x_{r+5})>x_r-x_{r+5}>0.
\end{equation}
If $x_{r+5}\ge0$, this trivially holds.
If $x_{r+5}<0$, it holds by the following argument.	Note that $p(x_r,x_{r+5})=x_{r+4}\ge0$ which means $\mu^3x_{r+5}-11\mu^2x_{r+5}+36\mu x_{r+5}-34x_{r+5}-4x_r\le0$. From this and the fact that $\mu<0.355$ (by Lemma~\ref{lem:muUpBound}), we have
	$4x_r\ge(\mu^3-11\mu^2+36\mu-34)x_{r+5}>-12x_{r+5}$, that is $x_r>-3x_{r+5}$, and thus \eqref{eq:xr+xr+5} follows.

We	now replace $H_4$ by $H'_4$ to obtain $\G'$ and
	define a vector $\x'$ on $V(\G')$ such that its components on $H'_4$ are as given in Figure~\ref{fig:H'4inG'} and on the rest of vertices agree with $\x$.	
\begin{figure}[H]
\centering
\begin{tikzpicture}[scale=1]			
			\vertex[fill] (r3) at (.8,0) [label=left:\footnotesize{$z_{r+1}$}] {};
			\vertex[fill] (r4) at (1.35,.5) [label=above:\footnotesize{$z_{r+2}$}] {};
			\vertex[fill] (r5) at (1.35,-.5) [] {};
			\vertex[fill] (r6) at (1.9,0) [] {};
			\vertex[fill] (r7) at (2.5,.5) [label=above:\footnotesize{$z_{r+4}$}] {};
			\vertex[fill] (r8) at (2.5,-.5) [] {};
			\vertex[fill] (r9) at (3.5,.5) [label=above:\footnotesize{$x_{r+5}$}] {};
			\vertex[fill] (r10) at (3.5,-.5) [] {};
	\tikzstyle{vertex}=[circle, draw, inner sep=0pt, minimum size=0pt]
    \vertex[] () at (1.93,-.13) [label=below:\footnotesize{$z_{r+3}$}] {};
			\path	
			(r4) edge (r3)
			(r3) edge (r5)
			(r4) edge (r5)
			(r4) edge (r6)
			(r5) edge (r6)
			(r7) edge (r6)
			(r8) edge (r6)
			(r7) edge (r9)
			(r8) edge (r7)
			(r10) edge (r7)
			(r8) edge (r9)
			(r8) edge (r10)
			(r9) edge (r10);
			\end{tikzpicture}
\caption{The subgraph $H'_4$ in $\G'$ and the components of $\x'$}
\label{fig:H'4inG'}
\end{figure}
We set
$$z_{r+1}=p(x_{r+5},x_r), \quad z_{r+2}=l(x_{r+5},x_r), \quad z_{r+3}=g(x_{r+5},x_r), \quad z_{r+4}=f(x_{r+5},x_r).$$
We have $\delta=\x'\1^\top=z_{r+1}+2z_{r+2}+z_{r+3}+2z_{r+4}-2x_{r+1}-x_{r+2}-2x_{r+3}-x_{r+4}$. By substituting the values of $x_{r+1},\ldots,x_{r+4},$ and $z_{r+1},\ldots,z_{r+4}$ in terms of $x_r,x_{r+5},\mu$ we obtain that
$$\delta=\frac{2(\mu-4)(\mu-5)(x_r-x_{r+5})}{\mu^4-14\mu^3+67\mu^2-126\mu+76}.$$
	Now, in the notation of  Lemma~\ref{rem:main1}, we have
	\begin{align*}
	h&=4(x_{r}-x_{r+1})^2+2\sum_{i=r+1}^{r+4}(x_i-x_{i+1})^2+2(x_{r+3}-x_{r+5})^2,  \\
	h'&=2(x_{r}-z_{r+1})^2+2(x_{r}-z_{r+2})^2+2 \sum_{i=r+1}^{r+3}(z_i-z_{i+1})^2+4(z_{r+4}-x_{r+5})^2.
	\end{align*}
	Moreover,
$$\epsilon=\|\x'\|^2-1-\frac{\delta^2}n=z_{r+1}^2+2z_{r+2}^2+z_{r+3}^2+2z_{r+4}^2-2x_{r+1}^2-x_{r+2}^2-2x_{r+3}^2- x_{r+4}^2-\frac{\delta^2}n.$$
	It follows that
\begin{align*}
h'-h-\epsilon\mu&=8(x_{r}x_{r+1}-8z_{r+4}x_{r+5})
+4\left(\left(x_{{r+4}}+x_{{r+5}} \right) x_{{r+3}}-z_{r+2}\left( x_{r}+z_{r+1} \right)\right)\\
&\quad-4(x_{r+2}^{2}+x_{r+4}^{2}-z_{r+1}^{2}-z_{r+3}^{2})
 -6(x_{r+1}^{2}+x_{r+3}^{2}-z_{r+2}^{2}-z_{r+4}^{2})\\
&\quad -4(x_{r}z_{r+1}-x_{r+2}x_{r+1}-x_{r+4}x_{r+5}+z_{r+3}z_{r+4})
 +4(x_{r+2}x_{r+3}-z_{r+2}z_{r+3})-\epsilon\mu\\
&=\Phi\frac{-2( x_r-x_{r+5})(\mu-4)(\mu-5)\mu}{(\mu^4-14\mu^3+67\mu^2-126\mu+76)^2n},
\end{align*}
where $$\Phi=(\mu^4-14\mu^3+67\mu^2-126\mu+76)n (x_r+x_{r+5})+(-2\mu^2+18\mu-40)(x_r-x_{r+5}).$$
The smallest zero of $t^4-14t^3+67t^2-126t+46$ is about $0.472$. Hence $\mu^4-14\mu^3+67\mu^2-126\mu+46>0$ because $\mu<0.355$, and thus
	\begin{equation}\label{eq:Phi2}
	\Phi>30n(x_k+x_{k+5})-40(x_k-x_{k+5}).
	\end{equation}
In view of \eqref{eq:xr+xr+5}, the right hand side of \eqref{eq:Phi2} is positive which means $h'-h-\epsilon\mu<0$.	The result now follows from Lemma~\ref{rem:main1}.

 (ii)	If  $17\le n\le20$, and $\G$ contains $D_0M_3$, then the right end block must be
	$\tilde D_i$, $0\le i\le3$, respectively, which leads to a contradiction as
	$\mu(\g_n)<\mu(D_0M_3\tilde D_i)$ for $17\le n\le20$. So we assume that $n\geq 21$ and thus  from Lemma~\ref{lem:muUpBound}, $\mu<0.091$.

	For a contradiction, assume that $\G$ contains $H_5=D_0M_3$ with the components of $\x$ as depicted in Figure \ref{figH5H'5}.	By using the eigen-equation \eqref{eq:eigenvector}, we can write $x_2,\ldots,x_6$ in terms of $x_1$ and $\mu$ as follows:
	\begin{align*}
	x_2&=\left(\frac{-\mu}{2}+1\right)x_1, \\
	x_3&=\left(\frac{\mu^2}{2}-3\mu+1\right)x_1, \\
	x_4&=\frac{-1}{4}\left(\mu^3-10\mu^2+24\mu-4\right)x_1,\\
	x_5&=\left(\frac{\mu^4-14\mu^3+62\mu^2-88\mu+12}{2(6-\mu)}\right)x_1,\\
	x_6&=\left(\frac{\mu^5-17\mu^4+102\mu^3-252\mu^2+216\mu-24}{4(\mu-6)}\right)x_1.
	\end{align*}
	\begin{figure}[H]
		\captionsetup[subfigure]{labelformat=empty}
		\centering
		\begin{tikzpicture}[scale=0.9]
		\vertex[fill] (r) at (0,0) [label=left:\footnotesize{$x_1$}] {};
		\vertex[fill] (r1) at (.5,.5) [label=above:\footnotesize{$x_1$}] {};
		\vertex[fill] (r2) at (.5,-.5) [label=below:\footnotesize{$x_1$}] {};
		\vertex[fill] (r3) at (1.5,.5) [label=above:\footnotesize{$x_2$}] {};
		\vertex[fill] (r4) at (1.5,-.5) [label=below:\footnotesize{$x_2$}] {};
		\vertex[fill] (r5) at (2,0) [label=above:\footnotesize{$x_3$}] {};
		\vertex[fill] (r6) at (2.5,.5) [label=above:\footnotesize{$x_4$}] {};
		\vertex[fill] (r7) at (2.5,-.5) [label=below:\footnotesize{$x_4$}] {};
		\vertex[fill] (r8) at (3,0) [label=above:\footnotesize{$x_5$}] {};
		\vertex[fill] (r9) at (3.5,.5) [label=above:\footnotesize{$x_6$}] {};
		\vertex[fill] (r10) at (3.5,-.5) [label=below:\footnotesize{$x_6$}] {};
		\vertex[fill] (r11) at (4,0) [] {};
		\tikzstyle{vertex}=[circle, draw, inner sep=0pt, minimum size=0pt]
		\vertex (s) at (4.2,.2) []{};
		\vertex (ss) at (4.2,-.2) []{};
      \vertex(sss) at (3.8,-1.1) [] {};
		\path
		(r) edge (r1)
		(r) edge (r2)
		(r) edge (r3)
		(r) edge (r4)
		(r1) edge (r2)
		(r1) edge (r3)
		(r1) edge (r4)
		(r2) edge (r3)
		(r2) edge (r4)
		(r5) edge (r4)
		(r5) edge (r3)
		(r5) edge (r6)
		(r5) edge (r7)
		(r6) edge (r7)
		(r6) edge (r8)
		(r6) edge (r9)
		(r7) edge (r8)
		(r7) edge (r10)
		(r9) edge (r8)
		(r8) edge (r10)
		(r9) edge (r10)
		(r9) edge (r11)
		(r10) edge (r11)
		(r10) edge (s)
		(r11) edge (ss);
		\end{tikzpicture}
		\qquad
\begin{tikzpicture}[scale=0.9]
\vertex[fill] (r1) at (0,.8) [label=above:\footnotesize{$z_1$}] {};
\vertex[fill] (r2) at (.8,1) [label=above:\footnotesize{$z_1$}] {};
\vertex[fill] (r3) at (0,0) [label=below:\footnotesize{$z_1$}] {};
\vertex[fill] (r4) at (.8,-.2) [label=below:\footnotesize{$z_1$}]{};
\vertex[fill] (r5) at (1.6,.8) [label=above:\footnotesize{$z_2$}] {};
\vertex[fill] (r6) at (1.6,0) [label=below:\footnotesize{$z_2$}] {};
\vertex[fill] (r7) at (2.3,.4) [label=above:\footnotesize{$z_3$}] {};
\vertex[fill] (r8) at (2.8,.9) [label=above:\footnotesize{$z_4$}] {};
\vertex[fill] (r9) at (2.8,-.1) [label=below:\footnotesize{$z_4$}]{};
\vertex[fill] (r10) at (3.8,.9) [label=above:\footnotesize{$x_6$}] {};
\vertex[fill] (r11) at (3.8,-.1) [label=below:\footnotesize{$x_6$}] {};
\vertex[fill] (r12) at (4.3,.4) [] {};
\tikzstyle{vertex}=[circle, draw, inner sep=0pt, minimum size=0pt]
\vertex (s) at (4.5,.6)[] {};
\vertex (ss) at (4.5,.2) [] {};
\path
(r5) edge (r2)
(r1) edge (r2)
(r1) edge (r5)
(r1) edge (r3)
(r1) edge (r4)
(r2) edge (r3)
(r2) edge (r4)
(r3) edge (r6)
(r4) edge (r3)
(r4) edge (r6)
(r5) edge (r6)
(r7) edge (r6)
(r5) edge (r7)
(r7) edge (r8)
(r9) edge (r7)
(r9) edge (r8)
(r10) edge (r8)
(r11) edge (r8)
(r10) edge (r9)
(r9) edge (r11)
(r10) edge (r11)
(r10) edge (r12)
(r11) edge (r12)
(r12) edge (s)
(r12) edge (ss);
\end{tikzpicture}
			\caption{The subgraphs $H_5$ (left) and $H'_5$ (right) and the components of $\x$ and $\x'$}
			\label{figH5H'5}
	\end{figure}
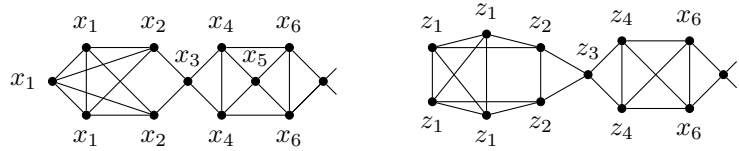
We obtain a contradiction by showing that if we replace $H_5$ by $H'_5$ (of Figure~\ref{figH5H'5}) to obtain $\G'$,
then $\mu(\G')<\mu$.
We define a vector $\x'$ on $V(\G')$ such that its components on $H'_5$ are as  given in Figure~\ref{figH5H'5}  and on the rest of vertices agree with $\x$. We set
\begin{equation*}
z_1=\omega x_1,~ z_2=\omega(1-\mu)x_1,~ z_3={\omega}(\mu^2-4\mu+1)x_1,~ z_4=\frac{-\omega}{2}(\mu^3-8\mu^2+15\mu-2)x_1,
\end{equation*}
where
$$\omega=\frac{\mu^5-17\mu^4+102\mu^3-252\mu^2+216\mu-24}{\mu^5-17\mu^4+103\mu^3-261\mu^2+238\mu-24}.$$
Taking into account that $\x\perp\1$, we see that
$$\delta=\x'\1^\top=4z_1+2z_2+z_3+2z_4-3x_1-2x_2-x_3-2x_4-x_5.$$
Also, since $\|\x\|=1$,
we obtain that
$$\|\x'\|^2=1+4z_1^2+2z_2^2+z_3^2+2z_4^2-3x_1^2-2x_2^2-x_3^2-2x_4^2-x_5^2.$$
 Further we have
 \begin{align*}
x' L(\G')\x'^\top &=\mu+ 2(z_1-z_2)^2+2\sum_{i=1}^{3}(z_i-z_{i+1})^2+4(z_4-x_6)^2\\
 &\quad-4(x_1-x_2)^2-2\sum_{i=1}^{5}(x_i-x_{i+1})^2-2(x_4-x_6)^2.
 \end{align*}
 As $\|\x'\|^2-\frac{\delta^2}{n}>0$, then $\frac{\x'L(\G')\x'^\top}{\|\x'\|^2-\frac{\delta^2}{n}}<\mu$ if and only  if $\x'L(\G')\x'^\top-\mu\|\x'\|^2+\mu\frac{\delta^2}{n}<0$. So, by substituting the values of $x_2,\ldots,x_6,  z_1,\ldots, z_4,$  and $\delta$ in terms of $x_1$ and $\mu$, we must show that
\begin{align*}
	&{\big{(}} (\mu^9-28\mu^8+326\mu^7-2042\mu^6+7429\mu^5-15770\mu^4+18492\mu^3-10320\mu^2+1800\mu-96)
	n \\
	&+(2\mu^7-46\mu^6+418\mu^5-1886\mu^4+4296\mu^3-4280\mu^2+1000\mu){\big{)}}
	{\big{(}}  \mu^5-18\mu^4+119\mu^3-348\mu^2 \\
	&+408\mu-100 {\big{)}}
	{\big{(}}  \mu-5{\big{)}}
	x_1^2 \mu^2 <0.
	\end{align*}
	Given that
	$n\geq21$ and  $\mu<0.091$, this is equivalent to
	\begin{align*}
	&21\mu^9-588\mu^8+6848\mu^7-42928\mu^6+156427\mu^5-333056\mu^4+392628\mu^3\\
	&-221000\mu^2+38800\mu-2016<0.
	\end{align*}
	This holds as $\mu<0.091$, and so by  Lemma~\ref{remark:delta},
	$\mu(\G')<\mu$, a contradiction.

  (iii)	For a contradiction, assume that $\G$ contains $H_6=D_2M_3$.
If $\G$  has only one middle block, and its right end block is different from $D_2$,
 then we are done by the previous lemmas. The right end block also cannot be $\tilde D_2$, since  $\mu(\g_{21})<\mu(D_2M_3\tilde D_2)$.
It follows that $\G$ has at least two middle blocks, which in turn implies that $n\geq26$ and so by  Lemma~\ref{lem:muUpBound},   $\mu<0.059$.

 Let the components of $\x$ on $H_6$ be as depicted in Figure \ref{figH6H'6}. By using the eigen-equation \eqref{eq:eigenvector}, we can write the weights of the vertices of $H_6$ in terms of $\mu$ and $x_8$ as follows:
\begin{align*}
x_1&=-4(2\mu^2-19\mu+42) \omega x_8,\\
x_2&=2(\mu-6)(\mu^2-7\mu+14)\omega x_8, \\
x_3&=4(\mu^3-12\mu^2+43\mu-42)wx_8, \\
x_4&=-2(\mu^4-16\mu^3+87\mu^2-176\mu+84)\omega x_8, \\
x_5&=2(\mu^5-19\mu^4+132\mu^3-400\mu^2+470\mu-84)\omega x_8, \\
x_6&=-(\mu^6-23\mu^5+206\mu^4-896\mu^3+1896\mu^2-1612\mu+168)\omega x_8, \\
x_7&=-2(\mu^6-21\mu^5+170\mu^4-662\mu^3+1244\mu^2-932\mu+84)\omega x_8,
\end{align*}
where $\omega=(\mu^7-24\mu^6+231\mu^5-1136\mu^4+2996\mu^3-4012\mu^2+2200\mu-168)^{-1}$.
	
	\begin{figure}[H]
		\captionsetup[subfigure]{labelformat=empty}
		\centering
		\begin{tikzpicture}[scale=0.9]
\vertex[fill] (r1) at (0,0) [label=below:\footnotesize{$x_1$}] {};
\vertex[fill] (r2) at (0,.8) [label=above:\footnotesize{$x_1$}] {};
\vertex[fill] (r3) at (.8,1) [label=above:\footnotesize{$x_2$}] {};
\vertex[fill] (r4) at (.8,-.2) [label=below:\footnotesize{$x_2$}] {};
\vertex[fill] (r5) at (1.3,.4) [label=above:\footnotesize{$x_3$}] {};
\vertex[fill] (r6) at (2,.8) [label=above:\footnotesize{$x_4$}] {};
\vertex[fill] (r7) at (2,0) [label=below:\footnotesize{$x_4$}] {}; {};
\vertex[fill] (r8) at (2.7,.4) [label=above:\footnotesize{$x_5$}] {};
\vertex[fill] (r9) at (3.2,-.1) [label=below:\footnotesize{$x_6$}] {};
\vertex[fill] (r10) at (3.2,.9) [label=above:\footnotesize{$x_6$}] {};
\vertex[fill] (r) at (3.7,.4) [label=above:\footnotesize{$x_7$}] {};
\vertex[fill] (r11) at (4.2,-.1)[label=below:\footnotesize{$x_8$}] {};
\vertex[fill] (r12) at (4.2,.9) [label=above:\footnotesize{$x_8$}] {};
\vertex[fill] (r13) at (4.7,.4)  [] {};
\tikzstyle{vertex}=[circle, draw, inner sep=0pt, minimum size=0pt]
\vertex (s) at (4.9,.6) [] {};
\vertex (ss) at (4.9,.2) [] {};
\vertex(sss) at (3.8,-.7) [] {};
\path
(r5) edge (r2)
(r1) edge (r2)
(r1) edge (r5)
(r1) edge (r3)
(r1) edge (r4)
(r2) edge (r3)
(r2) edge (r4)
(r3) edge (r6)
(r4) edge (r3)
(r4) edge (r7)
(r5) edge (r7)
(r5) edge (r6)
(r6) edge (r7)
(r7) edge (r8)
(r6) edge (r8)
(r8) edge (r9)
(r10) edge (r8)
(r9) edge (r10)
(r9) edge (r11)
(r9) edge (r)
(r10) edge (r)
(r11) edge (r)
(r12) edge (r)
(r10) edge (r12)
(r12) edge (r11)
(r11) edge (r13)
(r12) edge (r13)
(s) edge (r13)
(ss) edge (r13);
\end{tikzpicture}
		\qquad
\begin{tikzpicture}[scale=.9]
		\vertex[fill] (r1) at (0,1) [label=above:\footnotesize{$z_1$}] {};
		\vertex[fill] (r2) at (.8,1.2) [label=above:\footnotesize{$z_1$}] {};
		\vertex[fill] (r3) at (0,0) [label=below:\footnotesize{$z_1$}] {};
		\vertex[fill] (r4) at (.8,-.2) [label=below:\footnotesize{$z_1$}] {};
		\vertex[fill] (r5) at (1.6,1) [label=above:\footnotesize{$z_2$}] {};
		\vertex[fill] (r6) at (1.6,0) [label=below:\footnotesize{$z_2$}] {};
		\vertex[fill] (r7) at (2.6,1) [label=above:\footnotesize{$z_3$}] {};
		\vertex[fill] (r8) at (2.6,0) [label=below:\footnotesize{$z_3$}] {};
		\vertex[fill] (r9) at (3.1,.5) [label=above:\footnotesize{$z_4$}] {};
		\vertex[fill] (r10) at (3.6,1) [label=above:\footnotesize{$z_5$}] {};
		\vertex[fill] (r11) at (3.6,0) [label=below:\footnotesize{$z_5$}] {};
		\vertex[fill] (r12) at (4.6,1) [label=above:\footnotesize{$x_8$}] {};
		\vertex[fill] (r13) at (4.6,0) [label=below:\footnotesize{$x_8$}] {};
		\vertex[fill] (r14) at (5.1,.5) [] {};
		\tikzstyle{vertex}=[circle, draw, inner sep=0pt, minimum size=0pt]
		\vertex (s) at (5.3,.7) [] {};
		\vertex (ss) at (5.3,.3) [] {};
		\path
		(r5) edge (r2)
		(r1) edge (r2)
		(r1) edge (r5)
		(r1) edge (r3)
		(r1) edge (r4)
		(r2) edge (r3)
		(r2) edge (r4)
		(r3) edge (r6)
		(r4) edge (r3)
		(r4) edge (r6)
		(r5) edge (r8)
		(r5) edge (r7)
		(r6) edge (r8)
		(r6) edge (r7)
		(r7) edge (r8)
		(r9) edge (r7)
		(r9) edge (r8)
		(r10) edge (r9)
		(r11) edge (r9)
		(r11) edge (r10)
		(r10) edge (r12)
		(r10) edge (r13)
		(r12) edge (r11)
		(r11) edge (r13)
		(r12) edge (r13)
		(r14) edge (r13)
		(r12) edge (r14)
		(r14) edge (s)
		(r14) edge (ss)   ;
		\end{tikzpicture}
			\caption{The subgraphs $H_6$ (left) and $H'_6$ (right) and the components of $\x$ and $\x'$}
			\label{figH6H'6}
	\end{figure}
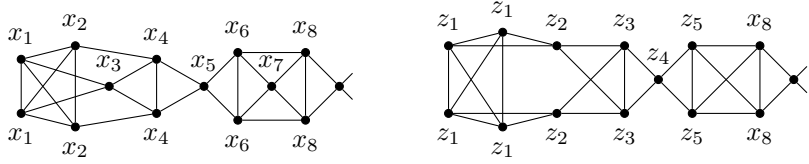
Now, we replace $H_6$ by $H'_6$ (of Figure~\ref{figH6H'6}) to obtain $\G'$.
We define a vector $\x'$ on $V(\G')$ such that its components on $H'_6$ are as given in Figure~\ref{figH6H'6}  and on the rest of vertices $\G'$ agree with $\x$. We specify the new components as follows:
\begin{align*}
z_1&=-8 \omega' x_8,\quad z_2=8 (\mu-1) \omega' x_8,\quad z_3=-4 (\mu^2-5\mu+2) {\omega'} x_8,\\
\quad z_4&=4 (\mu^3-8\mu^2+13\mu-2)\omega' x_8, \quad z_5=-2 (\mu^4-12\mu^3+43\mu^2-44\mu+4) \omega' x_8,
\end{align*}
where
$\omega'=( \mu^5-15\mu^4+77\mu^3-157\mu^2+110\mu-8)^{-1}.$
Taking into account that $\x\perp\1$, we see that
$$\delta=\x'\1^\top=4z_1+2z_2+2z_3+z_4+2z_5-2x_1-2x_2-x_3-2x_4-x_5-2x_6-x_7.$$
Also, since $\|\x\|=1$,
we obtain
$$\|\x'\|^2=1+4z_1^2+2z_2^2+2z_3^2+z_4^2+2z_5^2-2x_1^2-2x_2^2-x_3^2-2x_4^2-x_5^2-2x_6^2-x_7^2.$$
 Further we have
 \begin{align*}
x' L(\G')\x'^\top &=\mu+ 4(z_1-z_2)^2+4(z_2-z_3)^2+2(z_3-z_4)^2+2(z_4-z_5)^2+4(z_5-x_8)^2\\
&\quad-4(x_1-x_2)^2-2(x_1-x_3)^2-2(x_2-x_4)^2-2\sum_{i=3}^{7}(x_i-x_{i+1})^2-2(x_6-x_8)^2.
\end{align*}
 As $\|\x'\|^2-\frac{\delta^2}{n}>0$, we have $\frac{\x'L(\G')\x'^\top}{\|\x'\|^2-\frac{\delta^2}{n}}<\mu$ if and only  if $\x'L(\G')\x'^\top-\mu\|\x'\|^2+\mu\frac{\delta^2}{n}<0$. So, by  substituting the values of $x_1,\ldots,x_7,  z_1,\ldots, z_5,$  and $\delta$ in terms of $x_8$ and $\mu$, we need to show that
  \begin{equation}\label{eq:fracl}
{\left(p_1(\mu)n+p_2(\mu)\right)p_3(\mu)\mu^2 (\mu-6) x_8^2}<0,
\end{equation}
  in which
 \begin{align*}
p_1(t)&={t}^{12}-39\,{t}^{11}+668\,{t}^{10}-6606\,{t}^{9}+41701\,{t}^{8}-
175339\,{t}^{7}+497026\,{t}^{6}-939272\,{t}^{5}\\
&\quad +1140452\,{t}^{4}-823624\,{t}^{3}+300472\,{t}^{2}-36080\,t+1344,\\
p_2(t)&=2\,{t}^{10}-68\,{t}^{9}+992\,{t}^{8}-8100\,{t}^{7}+40446\,{t}^{6}-
126432\,{t}^{5}+242312\,{t}^{4}-264520\,{t}^{3}\\
&\quad +137816\,{t}^{2}-20208\,t, \\
p_3(t)&={t}^{8}-28\,{t}^{7}+328\,{t}^{6}-2082\,{t}^{5}+7731\,{t}^{4}-16830\,{t
}^{3}+20176\,{t}^{2}-11204\,t+1684.
  \end{align*}
The smallest root of $p_3(t)$ is about $0.227$. As $\mu<0.059$, we have that $p_3(\mu)>0$.
The smallest root of $p_1(t)$ is about $0.081$ and so $p_1(\mu)>0$. Since $n\ge 26$ we also have
$$p_1(\mu)n+p_2(\mu)\ge26p_1(\mu)+p_2(\mu)=p_4(\mu),$$
where
\begin{align*}
p_4(t)&=26\,{t}^{12}-1014\,{t}^{11}+17370\,{t}^{10}-171824\,{t}^{9}+1085218\,{
t}^{8}-4566914\,{t}^{7}+12963122\,{t}^{6}\\
&\quad -24547504\,{t}^{5}+29894064\,{t}^{4}-21678744\,{t}^{3}+7950088\,{t}^{2}-958288\,t+34944.
\end{align*}
The smallest root of $p_4(t)$ is about $0.070$ which means that $p_4(\mu)>0$. Therefore,   $(p_1(\mu)n+p_2(\mu))p_3(\mu)<0$, and so \eqref{eq:fracl} is established. Hence,  by  Lemma~\ref{remark:delta},  $\mu(\G')<\mu$ which is a contradiction.
\end{proof}

Now we are prepared to prove the desired result on the middle blocks of $\G$.

\begin{theorem}\label{thm:middleblock}
	Any middle block of $\G$ is $M_0$.
\end{theorem}
\begin{proof}
By Lemma~\ref{lem:M5orM'1}, any middle block of $\G$ is either $M_0$ or $M_3$.
For a contradiction, assume that some middle block $B$ of $\G$ is $M_3$.

First, assume that $B$ is a block next to an end block.
By Theorem~\ref{thm:endblocks}, either of the end block of $\G$ is one of $D_0,D_1,D_2,D_3$, or $D_4$.
By Lemma~\ref{lem:noM5}\,(ii),(iii), $D_0M_3$ and $D_2M_3$  are not possible.
Let $H\in\{D_1M_3,D_3M_3,D_4M_3\}$. Then $H$ contains an induced subgraph isomorphic to $H_4$ possessing the additional condition of  Lemma~\ref{lem:noM5}\,(i) on the neighbors of the left-most vertices. If the components of $\x$ on the first four vertices of $M_3$ in $H$ are non-negative, then we are done.
Otherwise, $\x$ is negative on the vertices of $\G$ starting from the middle vertex of $M_3$ in $H$.
If all the middle blocks of $\G$ are $M_3$, then we have a right end block $M_3\tilde D_i$ which is not possible by the above argument.
Otherwise, we have the subgraph $M_3M_0$ in $\G$. Under the above condition, $\tilde\G$ contains $H_4$ satisfying the conditions of  Lemma~\ref{lem:noM5}\,(i), which leads again to a contradiction.

Now, assume that $B$ is not a block next to an end block. So the block on its left must be an $M_0$ which means that $\G$ has a $M_0M_3$ subgraph. However,  $M_0M_3$ also contains an induced subgraph isomorphic to $H_4$. If the components of $\x$ on the first three vertices of $M_3$ in $B$ are non-negative, then we are done. Otherwise, similar to the above argument, we obtain a contradiction.
\end{proof}

\section*{Acknowledgements}
The second author carried out this work during a Humboldt Research Fellowship at the University of Hamburg. He thanks the Alexander von Humboldt-Stiftung for financial support.


\begin{thebibliography}{9}
\bibitem{Abdi} M. Abdi, E. Ghorbani, and W. Imrich, Regular graphs with minimum spectral gap, {\em European J. Combin.} {\bf95} (2021), 103328, 18 pp.
\bibitem{actt} S.G. Aksoy, F.R. Chung, M. Tait, and J. Tobin, The maximum relaxation time of a random walk, {\em Adv. in Appl. Math.} {\bf101} (2018), 1--14.
\bibitem{aldous2002reversible} D. Aldous and J. Fill,  {\em Reversible Markov Chains and Random Walks on Graphs}, University of California, Berkeley,  2002, available at \url{http://www.stat.berkeley.edu/~aldous/RWG/book.html}
\bibitem{Imrich} C. Brand, B. Guiduli, and W. Imrich, The characterization of cubic graphs with minimal eigenvalue
gap, {\em Croatica Chemica Acta} {\bf80} (2007), 193--201.
\bibitem{Bussemaker} F.C. Bussemaker, S. \v Cobelji\'c, D.M. Cvetkovi\' c, and J.J. Seidel, Computer investigation of cubic graph, Technical Report No. 76-WSK-01,
Technological University Eindhoven, (1976).
\bibitem{Bussemaker2} F.C. Bussemaker, S. \v Cobelji\'c, D.M. Cvetkovi\' c, and J.J. Seidel, Cubic graphs on $\leq 14$  vertices,
{\em J. Combin. Theory Ser. B} {\bf23} (1977), 234--235.
\bibitem{chung} F.R. Chung, {\em Spectral Graph Theory}, vol. 92, American Mathematical Society, 1997.
\bibitem{fiedler1973algebraic}
M. Fiedler, Algebraic connectivity of graphs, {\em Czechoslovak Math. J.} {\bf23} (1973), 298--305.
\bibitem{GuiduliThesis} B. Guiduli, {\em Spectral Extrema for Graphs},  Ph.D. Thesis,  University of Chicago, 1996.
\bibitem{Guiduli} B. Guiduli,  The structure of trivalent graphs with minimal eigenvalue gap, {\em J. Algebraic Combin.} {\bf6} (1997), 321--329.
\end{thebibliography}
\end{document}